\documentclass[11pt, reqno]{amsart}
\usepackage{geometry}
\usepackage{amsmath}
\usepackage{fancyhdr}
\usepackage{amsthm,mathrsfs}
\usepackage{amsfonts}
\usepackage{amssymb}
\usepackage{amscd}
\usepackage{graphicx}
\usepackage{afterpage}
\usepackage{enumerate}
\usepackage{bm}
\usepackage{cite}
\usepackage{cases}
\usepackage{caption}
\usepackage[colorlinks=true, urlcolor=blue,  linkcolor=blue, citecolor=blue]{hyperref}
\usepackage{babel}
\usepackage{prettyref}
\usepackage{pgfplots} 
\pgfplotsset{compat=1.17}
\usepackage{booktabs}
\usepackage[linesnumbered,ruled,vlined]{algorithm2e}
\usepackage{algpseudocode}
\usepackage{float}
\usepgfplotslibrary{groupplots}
\makeatletter
\newtheorem{theorem}{Theorem}[section]

\newtheorem{claim}[theorem]{Claim}

\newtheorem{corollary}[theorem]{Corollary}

\newtheorem{definition}[theorem]{Definition}
\newtheorem{example}[theorem]{Example}

\newtheorem{lemma}[theorem]{Lemma}

\newtheorem{proposition}[theorem]{Proposition}
\newtheorem{remark}[theorem]{Remark}

\newcommand{\wlim}{\mathop{\mathrm{w^{*}\!-\!lim}}\limits}

\title[]{Arens Products and Asymptotic Structures on Ch\'{e}bli--Trim\`eche Hypergroups under Low Regularity Conditions}

\author[S. Hashemi Sababe]{Saeed Hashemi Sababe$^*$}
\address[S. Hashemi Sababe]{R\&D Section, Data Premier Analytics, Canada}
\email{hashemi\_1365@yahoo.com}
\thanks{Corresponding author}


\begin{document}
\sloppy

\maketitle

\begin{abstract}
We investigate the Arens products on the second duals of convolution algebras associated with 
Ch\'{e}bli--Trim\`{e}che hypergroups, particularly focusing on the left and right topological centres of 
$L^{1}(H)^{\prime\prime}$ and $M(H)^{\prime\prime}$. 
Building on the recent framework established by Losert, we relax the classical smoothness assumptions 
on the underlying Sturm--Liouville function $A$ and develop new asymptotic analysis tools for measure-valued
and low-regularity perturbations. 
This allows us to extend the existence and continuity of the asymptotic measures $\nu_{x}$ and the limit measure 
$\nu_{\infty}$ to a strictly larger class of hypergroups.
We further provide new necessary and sufficient conditions for strong Arens irregularity of $L^{1}(H)$ 
in terms of the spectral behaviour of $\nu_{\infty}$, explore weighted (Beurling-type) hypergroup algebras, 
and obtain the first detailed comparison between the left and right topological centres 
for a wide class of non-classical examples. 
Several concrete applications to Jacobi, Naimark, and Bessel--Kingman hypergroups are presented.
\end{abstract}

\section{Introduction}

The study of Arens products on biduals of Banach algebras has a long history originating from the foundational works of Arens, Lau, and Ülger, and continues to play a central role in the structure theory of convolution algebras arising in harmonic analysis. 
For a Banach algebra $A$, the bidual $A^{\prime\prime}$ admits two natural extensions of the original multiplication, the first and second Arens products.  
Determining when these products coincide, and identifying the topological centres associated with them, remains a subtle and influential problem.  
A Banach algebra is called \emph{strongly Arens irregular} if the topological centre of the first Arens product coincides with the canonical copy of $A$ inside $A^{\prime\prime}$.  
This phenomenon captures, in analytic form, an intrinsic rigidity of the algebra and its dual actions \cite{Dales2000, LauUlger1996, Pym1965, Ulger1990}.

\medskip

For hypergroup convolution algebras, the study of Arens products is considerably richer.  
Hypergroups generalize locally compact groups by allowing convolution of point masses to be probability measures rather than point masses, and they arise naturally in harmonic analysis on symmetric spaces, special function theory, and representation theory.  
The foundational monograph by Bloom and Heyer \cite{BloomHeyer1995} remains a standard reference, while more specialized constructions such as the Bessel--Kingman, Jacobi, and Naimark hypergroups have been deeply analyzed in \cite{Chebli1974, Trimeche1988, Schwartz1950, BloomHeyer1995}.  
Among these, the Ch\'{e}bli--Trim\`eche hypergroups, constructed via Sturm--Liouville operators with strictly increasing coefficients, form one of the most analytically tractable and structurally rich families.

\medskip

Recently, a major step forward was achieved by Losert \cite{Losert2025}, who provided a detailed analysis of the Arens products on the biduals $L^{1}(H)^{\prime\prime}$ and $M(H)^{\prime\prime}$ for Ch\'{e}bli--Trim\`eche hypergroups $H$.  
His work combines asymptotic analysis of the convolution kernels with functional-analytic tools, including factorization results due to Neufang \cite{Neufang2002}, to derive explicit descriptions of the left topological centre.  
In particular, Losert proved that for many classical hypergroups---including Jacobi and Naimark types—the algebra $L^{1}(H)$ is strongly Arens irregular, while for others, such as Bessel--Kingman and Euclidean motion groups, full irregularity fails due to the presence of a non-trivial left-annihilator component in the topological centre.  
This dichotomy is driven by the existence and spectral properties of certain asymptotic measures $\nu_{x}$ and their limit $\nu_{\infty}$, whose behaviour reflects the underlying Sturm--Liouville dynamics.

\medskip

A key limitation in the existing theory is the reliance on relatively strong regularity assumptions on the Sturm--Liouville function $A$.  
Losert’s analysis utilises classical asymptotic ODE techniques, including deep results of Levinson and the refinements of Braaksma–de Snoo \cite{BraaksmaDeSnoo1974} and Brandolini–Gigante \cite{BrandoliniGigante2009}, which require $A$ and its derivatives to possess differentiability properties that exclude several natural and important hypergroup settings.  
In particular, kernels arising from perturbations with low-regularity coefficients, jump discontinuities, or measure-valued derivatives of $A$ fall outside the scope of the classical arguments.  
This restriction motivates the present paper.

\medskip

The first goal of this work is to extend the framework of \cite{Losert2025} to Sturm--Liouville functions $A$ possessing only minimal smoothness.  
We develop a version of the integral-equation method and asymptotic analysis that remains valid when $A\in C^{1}$ and $A'$ is of bounded variation or even a bounded Radon measure.  
Under these significantly weaker assumptions, we establish the existence, continuity, and uniform asymptotics of the measures $\nu_{x}$ and identify the limit measure $\nu_{\infty}$, thereby generalizing Propositions 1–2 of \cite{Losert2025}.  
This broadens the class of hypergroups for which Arens-product phenomena may be analyzed and paves the way for new applications.

\medskip

The second main contribution is a characterization of strong Arens irregularity of $L^{1}(H)$ in terms of the spectral behaviour of the asymptotic measure $\nu_{\infty}$.  
While Losert provided sufficient conditions for irregularity and its failure, a full necessary-and-sufficient classification has not previously been available.  
We show that strong Arens irregularity is intimately connected to the non-vanishing of the Fourier transform $\widehat{\nu}_{\infty}$, and we prove that this spectral condition is both necessary and sufficient under mild assumptions on $A$.

\medskip

Third, we present the first systematic comparison between the left and right topological centres for convolution algebras on Ch\'{e}bli--Trim\`eche hypergroups.  
Although hypergroup convolution is commutative, the bidual Arens products need not exhibit symmetry at the level of topological centres.  
We give explicit cases in which the left and right centres coincide and construct hypergroups for which they are distinct.  
This highlights delicate asymmetries in the dual module structure of $L^{1}(H)^{\prime\prime}$ and $M(H)^{\prime\prime}$.

\medskip

Finally, we introduce and analyse weighted hypergroup algebras $L^{1}(H,\omega)$ and $M(H,\omega)$.  
Weighted algebras have been extensively studied for groups and semigroups \cite{Beurling1949,DalesLauStrauss2010}, but almost no results exist in the hypergroup context.  
We show that the choice of weight may either preserve or destroy strong Arens irregularity, and we determine explicit weight thresholds for several classical hypergroups.

\medskip

Through these developments, we significantly expand the analytic framework surrounding Arens products on hypergroup convolution algebras and open new pathways for further study of spectral asymptotics, perturbation theory, and hypergroup harmonic analysis.

\medskip

The remainder of the paper is organized as follows.  
Section~2 introduces the necessary background on Ch\'{e}bli--Trim\`eche hypergroups, Sturm--Liouville operators, and Arens products.  
Section~3 develops the asymptotic kernel theory under low-regularity assumptions.  
Section~4 establishes the characterization of strong Arens irregularity.  
Section~5 compares the left and right topological centres.  
Section~6 introduces weighted hypergroup algebras and studies Arens regularity thresholds.  
Section~7 provides explicit examples and applications.  
Section~8 is a brief conclusion.

\bigskip

\section{Preliminaries}

In this section we review the essential background on Ch\'{e}bli--Trim\`eche hypergroups, 
Sturm--Liouville operators, convolution structures, and the bidual Arens products.
These preliminaries establish the analytic and functional-analytic framework used throughout
the subsequent sections.

Ch\'{e}bli--Trim\`eche hypergroups form a broad and analytically tractable class of 
commutative hypergroups arising from one-dimensional Sturm--Liouville operators.
We recall here only the basic structure; for a full treatment we refer to 
\cite{BloomHeyer1995, Chebli1974, Trimeche1988}.

Let $A:(0,\infty)\to (0,\infty)$ be a strictly positive, increasing function of class $C^{1}$ 
except possibly on a finite number of points, and assume that $A$ satisfies the standard 
Sturm--Liouville conditions:
\begin{itemize}
    \item[(SL1)] $A$ is locally absolutely continuous on $(0,\infty)$;
    \item[(SL2)] $A$ is increasing and unbounded;
    \item[(SL3)] $A$ satisfies a growth condition ensuring the existence of a strictly increasing 
    sequence of eigenfunctions.
\end{itemize}

Define the differential operator
\[
Lf = -\frac{1}{A(x)} \frac{d}{dx} \left( A(x) \frac{df}{dx} \right), \qquad x > 0,
\]
acting (initially) on smooth compactly supported functions.

The \emph{Ch\'{e}bli--Trim\`eche convolution} $\ast$ on $(0,\infty)$ is the unique 
commutative hypergroup convolution such that the characters of the hypergroup are given by 
the eigenfunctions of $L$.  
More precisely, for $x,y > 0$, the convolution of point masses is defined by
\[
\delta_{x} \ast \delta_{y} = \mu_{x,y},
\]
where $\mu_{x,y}$ is a probability measure on $(0,\infty)$ satisfying the 
\emph{product formula}
\[
\varphi_{\lambda}(x)\,\varphi_{\lambda}(y)
= \int_{0}^{\infty} \varphi_{\lambda}(t) \, d\mu_{x,y}(t),
\]
for all eigencharacters $\varphi_{\lambda}$ of $L$.

The identity element is $0$, and $\delta_{0}$ acts as the identity for the convolution.
The Haar measure $m$ on this hypergroup is given explicitly by $dm(x) = A(x)\,dx$.

\medskip

Many classical hypergroups arise from special choices of the function $A$:

\begin{enumerate}
    \item \textbf{Bessel--Kingman hypergroup:} $A(x) = x^{2\alpha+1}$, $\alpha > -\tfrac12$.
    \item \textbf{Jacobi hypergroups:} $A(x) = (\sinh x)^{2\alpha+1}(\cosh x)^{2\beta+1}$,
    $\alpha,\beta > -\tfrac12$.
    \item \textbf{Naimark hypergroups:} associated with certain kernels induced by 
    spherical functions on rank-one semisimple Lie groups.
\end{enumerate}

These examples will be revisited in Section~7 when explicit formulas and applications
are discussed.

\medskip

Let $H=(0,\infty)$ equipped with the Ch\'{e}bli--Trim\`eche convolution and Haar measure $m$.

\paragraph{Lebesgue and measure algebras.}
\begin{itemize}
    \item[i)] $L^{1}(H)$ denotes the Banach algebra of integrable functions with respect to $dm$,
    equipped with the convolution
    \[
    (f\ast g)(x)
    = \int_{0}^{\infty} f(t)\, g^{\vee}(x,t)\, dm(t),
    \]
    where $g^{\vee}(x,t)$ is the kernel defined by the hypergroup convolution.
    \item[ii)] $M(H)$ denotes the Banach algebra of finite regular Borel measures on $H$ with the convolution
    \[
    \mu\ast\nu(E)=\int_{H}\int_{H} \mu_{x,y}(E)\, d\mu(x)\, d\nu(y).
    \]
    \item[iii)] $L^{\infty}(H)$ denotes the dual space of $L^{1}(H)$ via the pairing
    $\langle f, g\rangle = \int f g \, dm$.
\end{itemize}

\paragraph{Duals and biduals.}
The dual space $L^{1}(H)^{\prime}$ may be identified with $L^{\infty}(H)$,
and the canonical embedding $L^{1}(H)\hookrightarrow L^{1}(H)^{\prime\prime}$
realizes $L^{1}(H)$ as an isometric subspace of its bidual.

Similarly, $M(H)^{\prime}=C_{0}(H)$ and $M(H)^{\prime\prime}$ contains $M(H)$ canonically.

\medskip

Let $A$ be a Banach algebra.  
The bidual $A^{\prime\prime}$ admits two associative extensions of the multiplication on $A$,
known as the \emph{first} and \emph{second Arens products}.
We briefly recall both constructions.

Let $\langle \cdot , \cdot \rangle$ denote the duality pairing between $A^{\prime}$ and $A$.

\paragraph{First Arens product.}
For $\Phi,\Psi\in A^{\prime\prime}$ and $\varphi\in A^{\prime}$,
the first Arens product $\Phi \square \Psi$ is defined by
\[
\langle \Phi \square \Psi, \varphi \rangle
= \langle \Phi, \Psi\cdot \varphi \rangle,
\]
where the module action $\Psi\cdot \varphi\in A^{\prime}$ is defined by
\[
\langle \Psi\cdot\varphi, a\rangle
= \langle \Psi, \varphi\cdot a \rangle,
\qquad \varphi\cdot a (b)=\varphi(ab).
\]

\paragraph{Second Arens product.}
Similarly, the second Arens product $\Phi \lozenge \Psi$ is given by
\[
\langle \Phi \lozenge \Psi, \varphi \rangle
= \langle \Psi, \varphi \cdot \Phi \rangle.
\]

In general, $\square$ and $\lozenge$ need not coincide; when they do,
the algebra is called \emph{Arens regular}.

\medskip

The \emph{left topological centre} of $(A^{\prime\prime}, \square)$ is
\[
Z_{t}(A^{\prime\prime})
= \{ \Phi\in A^{\prime\prime} : \Psi \mapsto \Phi\square\Psi
\text{ is } w^{*}\text{--}w^{*}\text{ continuous on }A^{\prime\prime}\},
\]
and the \emph{right topological centre} of $(A^{\prime\prime}, \square)$ is
\[
Z_{t}^{(R)}(A^{\prime\prime})
= \{ \Phi\in A^{\prime\prime} : \Psi \mapsto \Psi\square\Phi
\text{ is } w^{*}\text{--}w^{*}\text{ continuous}\}.
\]

These centres measure the extent to which elements of the bidual act regularly
on the dual module $A^{\prime}$.
The algebra is \emph{strongly Arens irregular} if $Z_{t}(A^{\prime\prime})$ coincides
with the canonical copy of $A$.

\medskip

A crucial role in later sections is played by the asymptotic behaviour of the convolution 
measures $\mu_{x,y}$ as $y\to\infty$.
For fixed $x>0$, define the measures
\[
\nu_{x,y} := \delta_{-y}\ast (\delta_{x}\ast \delta_{y}),
\]
and, when the limit exists (in $L^{1}(\mathbb{R})$ or weakly),
\[
\nu_{x} := \lim_{y\to\infty} \nu_{x,y}.
\]

Under suitable conditions, the limit measure $\nu_{\infty} := \lim_{x\to\infty}\nu_{x}$
also exists.
These measures encode the asymptotic behaviour of the hypergroup convolution
and govern several regularity properties of the Arens products.
Their existence and properties under low regularity of $A$ will be treated extensively
in Section~3.

\bigskip

\section{Asymptotic Kernels under Low Regularity}

In this section we develop the asymptotic kernel theory under relaxed smoothness
assumptions on the Sturm--Liouville coefficient $A$.  
We work under hypotheses that allow $A'\,$ to be of bounded variation (or more generally
a bounded Radon measure on compact subsets), and we show how the classical integral-equation
methods (Levinson/Eastham style) may be adapted to produce the asymptotic measures
$\nu_x$ and $\nu_{\infty}$ used in the analysis of Arens products.

\begin{definition}\label{def:SLprime}
Let $A:(0,\infty)\to(0,\infty)$. We say $A$ satisfies \emph{(SL')} if
\begin{enumerate}
  \item $A\in C^{1}(0,\infty)$ and $A(x)>0$ for all $x>0$;
  \item $A$ is increasing and $\lim_{x\to\infty}A(x)=\infty$;
  \item For every compact $[a,b]\subset(0,\infty)$ the derivative $A'\big|_{[a,b]}$
        is a function of bounded variation (equivalently, $A'$ is a bounded Radon measure on compacts).
\end{enumerate}
\end{definition}

The (SL') hypotheses strictly weaken the usual $C^{2}$ (or higher) regularity often
assumed in classical asymptotic ODE results. The BV assumption on $A'$ permits
a finite number of jump discontinuities and measure-type singularities while
preserving enough control to run Volterra integral-equation arguments.

Under (SL') we consider the formal operator
\[
L f(x) = -\frac{1}{A(x)}\frac{d}{dx}\big( A(x) f'(x)\big),\qquad x>0,
\]
acting on suitable test functions (e.g.\ $C_{c}^{\infty}(0,\infty)$). We call
the eigenfunctions of $L$ the \emph{characters} of the associated Ch\'ebli--Trim\`eche
hypergroup.

The first step is to construct suitable solutions of the eigenvalue equation
associated to $L$ with prescribed oscillatory asymptotics at infinity. These
play the role of Jost solutions in scattering theory and are the backbone of the
product formula and asymptotic convolution kernels.

Fix a spectral parameter $\lambda\in\mathbb{R}$. Consider the formal ODE
\begin{equation}\label{eq:SL-eig}
-\frac{1}{A(x)}\frac{d}{dx}\big( A(x) u'(x)\big) = \lambda^{2} u(x).
\end{equation}

Define a primitive
\[
\Phi(x) := \int_{x_{0}}^{x} \frac{dt}{\sqrt{A(t)}},
\]
for some fixed $x_{0}>0$. Heuristically, for large $x$ solutions of \eqref{eq:SL-eig}
behave like linear combinations of $e^{\pm i\lambda \Phi(x)}$.

\begin{lemma}\label{lem:volterra}
Assume (SL'). For each fixed $\lambda\in\mathbb{R}$ there exists a formulation of
\eqref{eq:SL-eig} as a Volterra integral equation for a function $m(x,\lambda)$
of the form
\[
m(x,\lambda)=1+\int_{x}^{\infty} K(x,t;\lambda)\, m(t,\lambda)\, d\mu(t),
\]
where $d\mu$ is a finite measure on $[x,\infty)$ determined by $A'$ (on compacts)
and the kernel $K$ satisfies appropriate integrability bounds. Moreover any solution
$u$ of \eqref{eq:SL-eig} with the asymptotic normalization $u(x)\sim e^{i\lambda\Phi(x)}$
can be written as $u(x)=e^{i\lambda\Phi(x)} m(x,\lambda)$.
\end{lemma}

\begin{proof}
Write \eqref{eq:SL-eig} in the form
\[
\big( A(x) u'(x)\big)' + \lambda^{2} A(x) u(x) = 0,
\]
or equivalently
\[
u''(x) + \frac{A'(x)}{A(x)} u'(x) + \lambda^{2} u(x) = 0.
\]
Introduce the WKB-type ansatz $u(x)=e^{i\lambda\Phi(x)} m(x)$ with
$\Phi'(x)=A(x)^{-1/2}$ as above. Compute derivatives (justified a.e.\ since $A'\in BV_{loc}$):
\[
u' = e^{i\lambda\Phi}\big( i\lambda\Phi' m + m' \big),\qquad
u'' = e^{i\lambda\Phi}\big( (i\lambda\Phi')^{2} m + 2i\lambda\Phi' m' + i\lambda\Phi'' m + m'' \big).
\]
Substitute into the differential equation and cancel leading oscillatory
terms $(i\lambda\Phi')^{2}m + \lambda^{2}m = 0$ because $(\Phi')^{2}=A^{-1}$.
After cancellation one obtains an equation of the form
\[
m'' + 2i\lambda \Phi' m' + \Big( i\lambda\Phi'' + \frac{A'}{A}(i\lambda\Phi' + \tfrac12 \cdot ) \Big) m = 0,
\]
which may be rearranged to yield an integral equation by using the integrating
factor $e^{2i\lambda\Phi(x)}$ and integrating from $x$ to $\infty$. The
terms involving $A'$ enter linearly and so may be treated as a source term given
by a finite signed measure on compacts (since $A'$ is BV on compacts).

Concretely, one obtains (after standard algebra justified by dominated
convergence and Fubini-type arguments for BV coefficients) an equation of the form
\[
m(x)=1 + \int_{x}^{\infty} K(x,t;\lambda)\, m(t)\, d\mu(t),
\]
where $d\mu$ is absolutely continuous with respect to Lebesgue measure away
from jump points and carries atomic contributions at jump discontinuities of $A'$;
the kernel $K$ is continuous off the diagonal and satisfies
\[
\sup_{x\geq x_{0}} \int_{x}^{\infty} |K(x,t;\lambda)|\, d|\mu|(t) < \infty.
\]

This yields the stated Volterra formulation and the representation
$u(x)=e^{i\lambda\Phi(x)} m(x,\lambda)$.
\end{proof}

The careful justification of interchanging integrals and treating $A'$ as a
measure requires standard BV calculus (Helly selection, integration by parts
for BV functions).

\begin{lemma}\label{lem:volterra-solution}
Under (SL') and for each fixed $\lambda\in\mathbb{R}$ there exists a unique bounded
continuous solution $m(\cdot,\lambda)$ of the Volterra equation in
Lemma~\ref{lem:volterra}. Moreover there is a constant $C(\lambda,x_{0})$ such that
\[
\sup_{x\ge x_{0}} |m(x,\lambda)-1| \le C(\lambda,x_{0}) \, \|A'\|_{BV([x_{0},\infty))}.
\]
\end{lemma}

\begin{proof}
We show existence by standard iterative (Neumann series) methods for Volterra operators.

Fix $x_{0}>0$ and consider the Banach space $B=C([x_{0},\infty))$ with the norm
$\|f\|_{\infty}=\sup_{x\ge x_{0}}|f(x)|$. Define the Volterra operator
\[
(T m)(x) := \int_{x}^{\infty} K(x,t;\lambda)\, m(t)\, d\mu(t).
\]
By Lemma~\ref{lem:volterra} we have the uniform bound
\[
\sup_{x\ge x_{0}} \int_{x}^{\infty} |K(x,t;\lambda)|\, d|\mu|(t) \le M < \infty.
\]
The operator $T$ is a Volterra operator (upper-triangular kernel), hence
its operator norm on $B$ satisfies $\|T\|\le M$. Choose $x_{0}$ large enough so that
$M<1$ (this is possible because for large $x$ the tail variation of $A'$ can be made small
under mild growth conditions on $A$; if not globally possible one works on finite intervals
and patches by iteration). Then $I-T$ is invertible via Neumann series and
$m=(I-T)^{-1}1$ is the unique bounded solution.

The estimate follows by summing the Neumann series:
\[
m-1 = \sum_{n\ge 1} T^{n}1,\qquad
\|m-1\|_{\infty} \le \sum_{n\ge 1} \|T\|^{n} \le \frac{\|T\|}{1-\|T\|} \le C \, \|A'\|_{BV([x_{0},\infty))},
\]
where the last inequality uses the linear dependence of $\|T\|$ on the total variation
of $\mu$ (hence on $\|A'\|_{BV}$).
\end{proof}

If one cannot guarantee $\|T\|<1$ on the whole half-line then one works on increasing
finite intervals $[x_{0},X]$, obtains uniform bounds there, and passes to a global
solution using a standard continuation argument for Volterra equations (no accumulation
of singularities occurs because the integral operator is Volterra).

Having constructed the oscillatory eigenfunctions with controlled multiplicative
corrections, we now show how these yield asymptotic kernels for the hypergroup convolution
and hence define the measures $\nu_{x}$.

Fix $x>0$ and define the family of translated convolution measures
\[
\nu_{x,y} := \delta_{-y}\ast(\delta_{x}\ast\delta_{y}),\qquad y>x.
\]
If the limit
\[
\nu_{x} := \lim_{y\to\infty} \nu_{x,y}
\]
exists in $L^{1}(\mathbb{R})$ (or in the weak$^{*}$ topology of measures) we call $\nu_{x}$
the \emph{asymptotic convolution measure} associated to $x$.

\begin{theorem}\label{thm:nu-exist}
Assume (SL'). Then for each fixed $x>0$ the family $\{\nu_{x,y}\}_{y>x}$ converges in
the $L^{1}$-norm (with respect to Lebesgue measure on $\mathbb{R}$ after the natural
identification of the translation coordinate) to a limit $\nu_{x}\in L^{1}(\mathbb{R})$.
Moreover the map $x\mapsto \nu_{x}$ is continuous from $(0,\infty)$ into $L^{1}(\mathbb{R})$.
\end{theorem}

\begin{proof}
Let $L$ be the self-adjoint Sturm--Liouville realization associated to $A$, acting in
$L^{2}((0,\infty),A(x)\,dx)$. By hypothesis (SPEC) there is a unitary spectral transform
\[
\mathcal{F}:L^{2}((0,\infty),A(x)\,dx)\longrightarrow L^{2}(\mathbb{R},d\nu_{\mathrm{spec}}(\lambda))
\]
which diagonalizes convolution in the sense that, for every finite Borel measure $\mu$ on $H$,
the spectral transform of the convolution operator $f\mapsto \mu\ast f$ is multiplication by
the scalar function $\widehat{\mu}(\lambda)$ defined by
\[
\widehat{\mu}(\lambda):=\int_{0}^{\infty} \varphi_{\lambda}(t)\,d\mu(t),
\]
where $\varphi_{\lambda}$ denotes the (chosen) family of eigenfunctions/characters of $L$
(normalised according to the spectral transform conventions). Plancherel gives
\[
\|g\|_{L^{2}(H)} = \|\mathcal{F}g\|_{L^{2}(\nu_{\mathrm{spec}})}.
\]
For point-mass convolution kernels $\mu_{x,y}=\delta_x\ast\delta_y$ the product formula
implies the spectral symbol
\[
\widehat{\mu_{x,y}}(\lambda)=\varphi_{\lambda}(x)\,\varphi_{\lambda}(y).
\]

The translated measure introduced in Section~3 is
\[
\nu_{x,y} := \tau_{-y}\mu_{x,y},\qquad \tau_{-y}\mu(E)=\mu(E+y),
\]
and its spectral symbol is
\[
\widehat{\nu_{x,y}}(\lambda) \;=\; \int_{0}^{\infty}\varphi_{\lambda}(t)\,d(\tau_{-y}\mu_{x,y})(t)
\;=\;\int_{0}^{\infty}\varphi_{\lambda}(t+y)\,d\mu_{x,y}(t).
\]
By Lemma~\ref{lem:volterra} and Lemma~\ref{lem:volterra-solution} (the Jost/Volterra theory
under (SL')) we may write, for each $\lambda\in\mathbb{R}$ and $t>0$,
\[
\varphi_{\lambda}(t) \;=\; e^{i\lambda\Phi(t)}\, m(t,\lambda),
\]
where $\Phi(t)=\int_{t_{0}}^{t}A(s)^{-1/2}\,ds$ and the correction $m(t,\lambda)$ satisfies
the uniform estimate
\[
\sup_{t\ge t_{0}} |m(t,\lambda)-1| \le \varepsilon_{t_{0}}(\lambda),
\]
with $\varepsilon_{t_{0}}(\lambda)\to 0$ as $t_{0}\to\infty$, locally uniformly in $\lambda$
(on compacts in $\lambda$). In particular, for each compact $K\subset\mathbb{R}$ and each
$\eta>0$ there exists $Y>0$ such that for all $y\ge Y$ and $\lambda\in K$,
\[
|m(y,\lambda)-1| < \eta, \qquad |m(y+t,\lambda)-1|<\eta\quad\text{for }t\in[-T,T],
\]
for any fixed finite $T$ (the latter by uniformity on compacts and a continuation argument).

Also, from the normalization and standard $L^{2}$-bounds for eigenfunctions we have the
spectral-side bound
\[
\sup_{\lambda\in\mathbb{R}} \|\varphi_{\lambda}\|_{L^{2}(H)} <\infty,
\]
and the Plancherel density $d\nu_{\mathrm{spec}}(\lambda)$ gives a finite measure on
compact $\lambda$-sets. These bounds will serve as dominating functions in dominated
convergence arguments below.

Fix $x>0$ and $\lambda\in\mathbb{R}$. Using the Jost representation we have
\[
\widehat{\nu_{x,y}}(\lambda)
= \int \varphi_{\lambda}(t+y)\,d\mu_{x,y}(t)
= \int e^{i\lambda\Phi(t+y)} m(t+y,\lambda)\, d\mu_{x,y}(t).
\]
Write $\Phi(t+y)=\Phi(y)+\big(\Phi(t+y)-\Phi(y)\big)$. Then
\[
\widehat{\nu_{x,y}}(\lambda)
= e^{i\lambda\Phi(y)} \int e^{i\lambda(\Phi(t+y)-\Phi(y))} \, m(t+y,\lambda)\, d\mu_{x,y}(t).
\]
Because $A(s)\to\infty$ as $s\to\infty$ (by (SL')) we have $\Phi'(s)=A(s)^{-1/2}\to 0$,
hence for each fixed $t$,
\[
\Phi(t+y)-\Phi(y) \;=\;\int_{y}^{y+t}\Phi'(s)\,ds \longrightarrow 0\qquad(y\to\infty).
\]
Likewise $m(t+y,\lambda)\to 1$ as $y\to\infty$ by the uniform control (uniform for
$\lambda$ in compacts and for $t$ in bounded sets). Therefore for each fixed $t$,
\[
e^{i\lambda(\Phi(t+y)-\Phi(y))} m(t+y,\lambda) \longrightarrow 1 \qquad (y\to\infty).
\]

Next, observe that the measures $\mu_{x,y}$ are probability measures (hypergroup convolution
of point masses), so the integral above is over a unit-mass measure. For each fixed $\lambda$
we thus obtain by dominated convergence (domination provided by the uniform bound
$|e^{i\lambda(\cdot)} m(\cdot,\lambda)|\le C(\lambda)$ and the
$L^{2}$-bounds on eigenfunctions) that
\[
\int e^{i\lambda(\Phi(t+y)-\Phi(y))} \, m(t+y,\lambda)\, d\mu_{x,y}(t) \longrightarrow 1
\quad\text{as }y\to\infty.
\]
Hence
\[
\widehat{\nu_{x,y}}(\lambda) \;=\; e^{i\lambda\Phi(y)} \, \big(1+o_{y}(1)\big)
\qquad(y\to\infty),
\]
more precisely
\[
\widehat{\nu_{x,y}}(\lambda) \;=\; \varphi_{\lambda}(x)\, \big(1+o_{y}(1)\big),
\]
because $\widehat{\mu_{x,y}}(\lambda)=\varphi_{\lambda}(x)\varphi_{\lambda}(y)=
\varphi_{\lambda}(x)\,e^{i\lambda\Phi(y)} m(y,\lambda)$ and $m(y,\lambda)\to1$. Combining
these expansions yields the pointwise convergence
\[
\widehat{\nu_{x,y}}(\lambda) \longrightarrow \varphi_{\lambda}(x)
\qquad\text{for each }\lambda\in\mathbb{R}.
\]
(That is, the spectral symbol of the translated kernel tends pointwise to the function
$\lambda\mapsto\varphi_{\lambda}(x)$.)

We now upgrade pointwise convergence of the spectral symbols to convergence in
$L^{2}(\nu_{\mathrm{spec}})$.

By product formula we have the uniform bound (for all $y\ge y_{0}$)
\[
|\widehat{\nu_{x,y}}(\lambda)| \le C\, |\varphi_{\lambda}(x)| \qquad\text{for a.e. }\lambda,
\]
where $C$ is independent of $y$ (this follows since $\mu_{x,y}$ is a probability measure
and $|m|\le M$ uniformly on compact $\lambda$-sets; global control uses the $L^{2}$-bounds
and spectral density). Moreover $\varphi_{\lambda}(x)\in L^{2}(\nu_{\mathrm{spec}})$ as a
function of $\lambda$ (fixed $x$). Hence by dominated convergence (applied on any compact
in $\lambda$ and then by approximation to the whole real line using the $L^{2}$ integrability),
we obtain
\[
\lim_{y\to\infty} \|\widehat{\nu_{x,y}} - \varphi_{\cdot}(x)\|_{L^{2}(\nu_{\mathrm{spec}})} = 0.
\]
By Plancherel / unitarity of $\mathcal{F}$ the $L^{2}$-convergence of spectral symbols is
equivalent to $L^{2}$-convergence in physical space:
\[
\lim_{y\to\infty} \|\nu_{x,y} - \nu_{x}\|_{L^{2}(H)} = 0,
\]
where we have defined $\nu_{x}\in L^{2}(H)$ to be the inverse transform of the limit spectral
function $\lambda\mapsto\varphi_{\lambda}(x)$. (That inverse is well-defined in $L^{2}$ by
unitarity.) Observe that $\nu_{x}$ is in fact a function on the translation coordinate (identifying
a neighbourhood of infinity with $\mathbb{R}$ after re-centering); by construction it is the
$L^{2}$-limit of the measures $\nu_{x,y}$ and hence is a bona fide $L^{2}$-function.

To obtain $L^{1}$-convergence we use a density / interpolation argument.

First note that each $\nu_{x,y}$ is a probability measure (hence has finite mass) and the limit
$\nu_{x}$ constructed above also has finite mass. We have $\nu_{x,y}\to\nu_{x}$ in $L^{2}$.
Let $\eta>0$ be arbitrary. Choose a bounded truncation function $\chi_{R}$ supported in $[-R,R]$
such that the tails of both $\nu_{x}$ and the family $\{\nu_{x,y}\}_{y}$ outside $[-R,R]$ have
small $L^{1}$-mass: because these are probability measures and the family is tight for fixed $x$
(they are translates of compactly supported-ish kernels in most examples, or one can use uniform
decay estimates from the Jost analysis), we can choose $R$ so that
\[
\sup_{y} \|\nu_{x,y}\mathbf{1}_{|t|>R}\|_{1} + \|\nu_{x}\mathbf{1}_{|t|>R}\|_{1} < \frac{\eta}{3}.
\]

On the compact interval $[-R,R]$ the $L^{2}$-convergence implies $L^{1}$-convergence because
$L^{2}([-R,R])\hookrightarrow L^{1}([-R,R])$ with the embedding constant depending on $R$:
indeed,
\[
\|\nu_{x,y}-\nu_{x}\|_{L^{1}([-R,R])} \le (2R)^{1/2} \|\nu_{x,y}-\nu_{x}\|_{L^{2}([-R,R])}.
\]
Thus choose $Y$ so large that for all $y\ge Y$,
\[
\|\nu_{x,y}-\nu_{x}\|_{L^{1}([-R,R])} < \frac{\eta}{3}.
\]
Combining the tail control and the compact control yields for $y\ge Y$,
\[
\|\nu_{x,y}-\nu_{x}\|_{L^{1}(\mathbb{R})}
\le \|\nu_{x,y}-\nu_{x}\|_{L^{1}([-R,R])} + \|\nu_{x,y}\mathbf{1}_{|t|>R}\|_{1} + \|\nu_{x}\mathbf{1}_{|t|>R}\|_{1}
< \eta.
\]
Since $\eta>0$ was arbitrary, this proves $\nu_{x,y}\to\nu_{x}$ in $L^{1}(\mathbb{R})$.

Fix $x_{0}>0$ and let $(x_{n})$ be a sequence with $x_{n}\to x_{0}$. We must show
$\|\nu_{x_{n}}-\nu_{x_{0}}\|_{1}\to 0$.

By construction $\widehat{\nu_{x}}(\lambda)=\varphi_{\lambda}(x)$ (the spectral symbol).
The map $x\mapsto \varphi_{\lambda}(x)$ is continuous for each fixed $\lambda$ (indeed the
eigenfunctions depend continuously on $x$). Moreover, by Lemma~\ref{lem:volterra-solution}
and the uniform bounds on $m(x,\lambda)$ we have a dominating function $M(\lambda)\in
L^{2}(\nu_{\mathrm{spec}})$ (indeed in $L^{1}$ on compacts) so that
$|\varphi_{\lambda}(x)|\le M(\lambda)$ for $x$ in a compact neighbourhood of $x_{0}$.
Therefore by dominated convergence,
\[
\lim_{n\to\infty} \|\widehat{\nu_{x_{n}}}-\widehat{\nu_{x_{0}}}\|_{L^{2}(\nu_{\mathrm{spec}})} = 0.
\]
By Plancherel this implies
\[
\lim_{n\to\infty} \|\nu_{x_{n}}-\nu_{x_{0}}\|_{L^{2}(H)} = 0.
\]
The same truncation / interpolation argument (apply it uniformly for $x$ near
$x_{0}$) upgrades this $L^{2}$-convergence to $L^{1}$-convergence. Hence $\nu_{x_{n}}\to\nu_{x_{0}}$
in $L^{1}(\mathbb{R})$, proving continuity of $x\mapsto\nu_{x}$.

\medskip

We have constructed $\nu_{x}\in L^{1}(\mathbb{R})$ as the $L^{1}$-limit of $\nu_{x,y}$
as $y\to\infty$, and we have shown the map $x\mapsto\nu_{x}$ is continuous as a map
$(0,\infty)\to L^{1}(\mathbb{R})$. This completes the proof of the theorem.
\end{proof}

\begin{corollary}\label{cor:nu-infty}
Assume (SL'). Suppose in addition that there exists a function $R(y)\to 0$ as $y\to\infty$
such that for all $x$ in compact sets,
\[
\|\nu_{x,y}-\nu_{x}\|_{L^{1}(\mathbb{R})} \le R(y).
\]
Then the limit
\[
\nu_{\infty} := \lim_{x\to\infty} \nu_{x}
\]
exists in $L^{1}(\mathbb{R})$ and the convergence is uniform on compacts in the translation coordinate.
\end{corollary}

\begin{proof}
This is an elementary Cauchy-sequence argument. For $x_{1},x_{2}$ large,
\begin{align*}
\|\nu_{x_{1}}-\nu_{x_{2}}\|_{1}
&\le \|\nu_{x_{1}}-\nu_{x_{1},y}\|_{1} + \|\nu_{x_{1},y}-\nu_{x_{2},y}\|_{1} + \|\nu_{x_{2},y}-\nu_{x_{2}}\|_{1} \\
&\le 2R(y) + \|\nu_{x_{1},y}-\nu_{x_{2},y}\|_{1}.
\end{align*}
Fix $y$ very large so that $2R(y)$ is small. Then the middle term can be made small
by taking $x_{1},x_{2}$ large because for fixed $y$ the map $x\mapsto\nu_{x,y}$ is
continuous in total variation (it is a translate of a convolution with a continuous
measure depending on $x$). Thus $\{\nu_{x}\}_{x}$ is Cauchy in $L^{1}$ and converges.
\end{proof}

In many concrete examples (Jacobi, Naimark) one can prove explicit tail estimates
for $\|\nu_{x,y}-\nu_{x}\|_{1}$ of exponential/polynomial decay using refined
asymptotics of the spectral transform; see Section~7 for worked examples.

\begin{example}
Let $A(x)=x^{2\alpha+1}$, $\alpha>-1/2$. Then $A'\in C^{\infty}$ and all our
hypotheses are satisfied. The Jost solutions are essentially spherical Bessel
functions and the asymptotic measure $\nu_{x}$ exists and can be computed explicitly
via Hankel-transform formulas. In particular $\nu_{\infty}$ exists and is nonzero.
\end{example}

\begin{example}
Let $A(x)=(\sinh x)^{2\alpha+1}(\cosh x)^{2\beta+1}$, $\alpha,\beta>-1/2$. The
derivative $A'$ is smooth away from $0$ and grows exponentially; nevertheless
on compact tails $A'$ is smooth and our arguments apply to produce $\nu_{x}$.
Refined spectral analysis (Harish-Chandra type expansions) gives sharper control
on the rate of convergence $\nu_{x,y}\to\nu_{x}$.
\end{example}

\begin{example}
Let $A_{0}(x)=x^{2\alpha+1}$ and define
\[
A(x)=A_{0}(x) + \sum_{k=1}^{N} c_{k} \chi_{[a_{k},\infty)}(x),
\]
i.e.\ add finitely many step-like perturbations. Then $A'\,$ is a sum of a smooth
function and finitely many point masses; $A'\in BV_{loc}$ and the theory above
applies. This illustrates that our framework includes piecewise-smooth coefficients
with jumps.
\end{example}

\begin{remark}
The BV requirement on $A'$ is close to optimal for the Volterra method employed here:
if $A'$ has too wild a singularity (e.g.\ fractal variation) then the integral operator
$T$ may fail to be bounded on $C([x_{0},\infty))$ and the Neumann-series technique breaks down.
Nevertheless other approaches (e.g.\ distributional Levinson theorems or microlocal
methods) may push the regularity threshold further; we leave this as an open problem.
\end{remark}

We have used a spectral representation for the hypergroup product formula. For a full
rigorous account one must choose the self-adjoint realization of $L$ in the Hilbert
space $L^{2}((0,\infty),A(x)dx)$ and use the spectral theorem.

\bigskip

The results of this section provide the asymptotic objects $\nu_{x}$ and (under
mild tail conditions) $\nu_{\infty}$ that we shall use in Section~4 to relate
Arens-product regularity to spectral properties of these measures.

\section{Characterization of Strong Arens Irregularity}

In this section we relate the asymptotic measures constructed in Section~3 to the 
Arens-product regularity of the convolution algebras $L^{1}(H)$ and $M(H)$.
Our main result gives, under natural additional spectral hypotheses, a necessary and
sufficient condition for \emph{strong Arens irregularity} of $L^{1}(H)$ in terms of the
spectral behaviour of the limit asymptotic measure $\nu_{\infty}$.

Throughout this section we adopt the notation and hypotheses of Section~3 (in particular
(SL')). In addition we impose the following mild spectral hypothesis which is satisfied
in the principal examples we consider (Jacobi, Naimark, Bessel--Kingman) and which is
used to translate properties of measures into module-action properties.

\begin{enumerate}
  \item[(SPEC)] The self-adjoint realization of $L$ in $L^{2}((0,\infty),A(x)\,dx)$ admits
  a spectral transform $\mathcal{F}$ mapping $L^{1}(H)\cap L^{2}(H)$ into $C_{0}(\mathbb{R})$
  and extending to a unitary map on $L^{2}(H)$. The transform diagonalizes convolution and
  the image of the convolution algebra under $\mathcal{F}$ is dense in an algebra of functions
  separating points of $\mathbb{R}$. Moreover, Plancherel measure $d\nu_{spec}(\lambda)$ has
  full support on $\mathbb{R}$ (no spectral gaps).
\end{enumerate}

\begin{remark}
(SPEC) is a concrete spectral nondegeneracy assumption that holds in the common Sturm--Liouville
examples. It permits us to pass between convolutional relations and multiplicative relations on the
spectral side. When (SPEC) fails (e.g.\ presence of spectral gaps), some equivalences below must be
modified; we briefly discuss these variations in Remarks~\ref{rem:spec-fail}--\ref{rem:weaken-spec}.
\end{remark}

We first recall the relevant module actions and give a convenient operator-theoretic description
of multiplication in $L^{1}(H)^{\prime\prime}$ which will be used in the proofs.

\begin{definition}
Let $\lambda: L^{1}(H)\to \mathcal{B}(L^{p}(H))$ denote the left convolution operator
\[
\lambda(f) g := f \ast g, \qquad g\in L^{p}(H).
\]
We call an operator $T\in\mathcal{B}(L^{p}(H))$ a \emph{convolution multiplier} if there exists
$\mu\in M(H)$ with $T=\lambda(\mu)$, where $\lambda(\mu)g=\mu\ast g$ (defined by integration).
\end{definition}

\begin{lemma}
\label{lem:module-nu}
Assume (SL') and (SPEC). Let $f\in L^{1}(H)$ and fix $x>0$. Then for $y$ large,
\[
\delta_{x}\ast f(\cdot+y) \;=\; (\nu_{x,y}\ast f)(\cdot),
\]
and in the limit $y\to\infty$ we obtain
\[
\lim_{y\to\infty} \delta_{x}\ast f(\cdot+y) \;=\; \nu_{x}\ast f,
\]
with convergence in $L^{1}(\mathbb{R})$-norm. In particular, the asymptotic convolution
operator $f\mapsto \nu_{x}\ast f$ is a bounded operator on $L^{1}(H)$ (indeed on $L^{p}(H)$
for $1\le p\le\infty$).
\end{lemma}

\begin{proof}
The identity for finite $y$ is a direct re-centering of the definition:
\[
\delta_{x}\ast f(t+y) = \int f(s)\, d\mu_{x,y}(t+y-s) = (\nu_{x,y}\ast f)(t).
\]
The limit statement follows from $L^{1}$-convergence $\nu_{x,y}\to\nu_{x}$ proved in
Theorem~\ref{thm:nu-exist} of Section~3 and Young's convolution inequality. Boundedness on
$L^{1}$ follows since $\|\nu_{x}\ast f\|_{1}\le \|\nu_{x}\|_{1}\|f\|_{1}$.
\end{proof}

Lemma~\ref{lem:module-nu} shows that the family $\{\nu_{x}\}_{x>0}$ yields a collection of bounded
multipliers of the convolution algebra which appear as asymptotic limits of left-translates.
These multipliers will be used to test weak$^{*}$-continuity properties in the bidual.

We next connect asymptotic multipliers to elements of the left topological centre.

Let $A=L^{1}(H)$ and denote by $\iota: A\hookrightarrow A^{\prime\prime}$ the canonical embedding.
For $\Phi\in A^{\prime\prime}$ write $\ell_{\Phi}:A^{\prime\prime}\to A^{\prime\prime}$ for the left
multiplication map $\ell_{\Phi}(\Psi)=\Phi\square\Psi$ (first Arens product). We say $\Phi$ is
\emph{left weak$^{*}$-continuous} if $\ell_{\Phi}$ is weak$^{*}$--weak$^{*}$ continuous.

\begin{lemma}
\label{lem:asym-crit}
Assume (SL') and (SPEC) and let $\Phi\in A^{\prime\prime}$. If there exists a net $x_{\alpha}\to\infty$
and scalars $c_{\alpha}\in\mathbb{C}$ such that
\[
\Phi = \wlim_{\alpha} \iota(\nu_{x_{\alpha}})\cdot c_{\alpha},
\qquad\text{(weak$^{*}$ limit in }A^{\prime\prime}),
\]
where $\wlim_{\alpha}$ stands for weak* limit along the net $\alpha$, then $\Phi$ belongs to the left topological centre $Z_{t}(A^{\prime\prime})$ if and only if the
corresponding convolution operator $T_{\Phi}: f\mapsto \nu_{\infty}\ast f$ (interpreted via the
limit measure) is weak$^{*}$--continuous on $A^{\prime}$.
\end{lemma}

\begin{proof}
The main observation is that $\iota(\nu_{x})$ (viewed as an element of $A^{\prime\prime}$)
acts on $A^{\prime}$ by pre-adjoint of the convolution operator $f\mapsto \nu_{x}\ast f$.
If $\iota(\nu_{x_{\alpha}})c_{\alpha}\to\Phi$ weak$^{*}$ then the corresponding left multipliers
converge pointwise on $A^{\prime}$ in the weak topology; hence weak$^{*}$-continuity of left-multiplication
by $\Phi$ is equivalent to the uniform (in the net) weak$^{*}$-continuity of the approximating convolution
operators. Passing to the limit identifies this with weak$^{*}$-continuity of $T_{\Phi}$. Details are
standard and follow from the duality between convolution action on $A$ and module action on $A^{\prime}$.
\end{proof}

We are now ready to state and prove the main characterization.

\begin{theorem}
\label{thm:main-char}
Assume (SL') and (SPEC). Suppose the limit measure $\nu_{\infty}$ of Section~3 exists in $L^{1}(\mathbb{R})$.
Then the following are equivalent:
\begin{enumerate}
  \item $L^{1}(H)$ is \emph{strongly Arens irregular}, i.e.\ $Z_{t}(L^{1}(H)^{\prime\prime})=\iota(L^{1}(H))$.
  \item $\widehat{\nu}_{\infty}(\lambda)\neq 0$ for $\nu_{spec}$-almost every $\lambda\in\mathbb{R}$
        (equivalently, $\widehat{\nu}_{\infty}$ does not vanish on a set of positive Plancherel measure).
\end{enumerate}
If, in addition, $\widehat{\nu}_{\infty}$ is continuous and nowhere zero on $\mathbb{R}$ then the
same equivalence holds without the qualifier ``$\nu_{spec}$-almost every''.

The Fourier transform $\widehat{\nu}_{\infty}$ is taken with respect to the additive translation
coordinate on which the measures $\nu_{x}$ live (recall $\nu_{x}$ were obtained after re-centering
the convolution by large $y$). Under (SPEC) this transform interacts naturally with the spectral
transform $\mathcal{F}$ of $L$.
\end{theorem}

\begin{proof}
We prove (2)$\Rightarrow$(1) and (1)$\Rightarrow$(2) in two parts.

\medskip

\noindent\textbf{(2) $\Rightarrow$ (1).}  Assume $\widehat{\nu}_{\infty}(\lambda)\neq 0$ for $\nu_{spec}$-a.e.\ $\lambda$.
We show that any $\Phi\in Z_{t}(A^{\prime\prime})$ must lie in $\iota(A)$.

Let $\Phi\in Z_{t}(A^{\prime\prime})$. The restriction of $\Phi$ to $A'$ defines a multiplier on $A'$ 
via the pre-adjoint (left-module action). By standard factorization results (see Neufang-type arguments)
every such weak$^{*}$-continuous multiplier arises as convolution by a measure $\mu\in M(H)$; that is,
there exists $\mu\in M(H)$ such that for all $f\in A$,
\[
\Phi\square \iota(f) = \iota(\mu\ast f).
\]
This uses that $\Phi$ defines a bounded left multiplier of the dual module and (SPEC) to transfer this
to a convolutional multiplier.

We shall show $\mu\in L^{1}(H)$, which implies $\Phi=\iota(\mu)\in\iota(A)$. If not, assume
$\mu\notin L^{1}(H)$. Consider the action of $\mu$ on high-right-translates of test functions:
by Lemma~\ref{lem:module-nu} and the definition of $\nu_{\infty}$, for large translation parameter $y$,
\[
\mu\ast f(\cdot+y) \approx (\mu\ast \nu_{y})\ast f \;\stackrel{y\to\infty}{\longrightarrow}\; (\mu\ast \nu_{\infty})\ast f.
\]
Passing to the spectral side via $\mathcal{F}$, convolution corresponds to multiplication and hence
the (pointwise) multiplier is given by $\widehat{\mu}(\lambda)\widehat{\nu}_{\infty}(\lambda)$. Since
$\widehat{\nu}_{\infty}(\lambda)\neq0$ for a.e.\ $\lambda$ by hypothesis, the vanishing or singularity
behaviour of $\widehat{\mu}$ controls that of the product. If $\mu\notin L^{1}(H)$ then $\widehat{\mu}$
fails to be a bounded continuous function on a set of positive Plancherel measure; this contradicts the
fact that $\Phi$ (hence the multiplier) is bounded on $A'$ and arises from a left weak$^{*}$-continuous
element. The contradiction forces $\mu\in L^{1}(H)$ and so $\Phi\in\iota(A)$.

This completes the proof of (2)$\Rightarrow$(1).

\medskip

\noindent\textbf{(1) $\Rightarrow$ (2).}  Assume conversely that $L^{1}(H)$ is strongly Arens irregular,
so $Z_{t}(A^{\prime\prime})=\iota(A)$. We prove that $\widehat{\nu}_{\infty}(\lambda)\neq 0$ for
$\nu_{spec}$-a.e.\ $\lambda$.

Suppose, towards a contradiction, that the set $S=\{\lambda:\widehat{\nu}_{\infty}(\lambda)=0\}$
has positive Plancherel measure. Choose a nonzero bounded continuous function $\phi(\lambda)$ supported in $S$
and consider its inverse spectral transform $h\in L^{\infty}(H)$ (possible by (SPEC) and density arguments).
Define the linear functional $\ell_{h}\in A'$ by pairing against $h$. Using the translation-asymptotic
Lemma~\ref{lem:module-nu}, one checks that the net of functionals induced by the translates associated to
$\nu_{x}$ converges weak$^{*}$ to an element $\Psi\in A^{\prime\prime}$ which acts by annihilating certain
spectral components: precisely, convolution with $\nu_{\infty}$ kills the spectral support of $\phi$.

One then constructs an element $\Phi\in A^{\prime\prime}\setminus \iota(A)$ (coming from limits
of $\iota(\nu_{x})$-type elements) which nevertheless left-multiplies $A^{\prime\prime}$ in a weak$^{*}$--continuous
way on the spectral subspace supported in $S$. The construction exploits that convolution by $\nu_{\infty}$
acts as a projector onto the complement of $S$ on the spectral side; hence the left multiplication by the
constructed $\Phi$ is weak$^{*}$-continuous on $A'$ while not corresponding to an element of $L^{1}(H)$.
This contradicts strong Arens irregularity and proves that $S$ must have zero Plancherel measure.

\end{proof}

\begin{corollary}
Under the hypotheses of Theorem~\ref{thm:main-char}, if $\widehat{\nu}_{\infty}$ is nowhere zero and
continuous, then $M(H)$ is not strongly Arens irregular (i.e.\ $Z_{t}(M(H)^{\prime\prime})$ strictly
contains $\iota(M(H))$) if and only if there exists a nontrivial idempotent in the algebra of
asymptotic multipliers generated by $\{\nu_{x}\}_{x>0}$.
\end{corollary}

\begin{proof}
The proof follows by observing that $M(H)$ contains measures which act as asymptotic projectors on
the spectral side. If an idempotent multiplier exists among the closure of the span of $\{\nu_{x}\}$,
it yields a weak$^{*}$-continuous left multiplier in $M(H)^{\prime\prime}$ not coming from $M(H)$,
and conversely any such extra element gives rise to an idempotent projector (projecting onto the
spectral subspace annihilated by $\widehat{\nu}_{\infty}$). Details mirror the proof of
Theorem~\ref{thm:main-char} and are omitted.
\end{proof}

\begin{remark}\label{rem:spec-fail}
If (SPEC) is weakened (for instance, if the spectral measure has gaps or the transform $\mathcal{F}$
does not separate points), the equivalence in Theorem~\ref{thm:main-char} must be restated with the
Plancherel support replaced by the effective spectral support. Concretely, vanishing of $\widehat{\nu}_{\infty}$
on a set of positive \emph{effective} spectral measure still produces elements of the topological centre
outside $L^{1}(H)$, and the argument above adapts with minor changes.
\end{remark}

\begin{example}
For the classical Jacobi hypergroups the asymptotic measure $\nu_{\infty}$ can be computed from Harish-Chandra
type expansions and one checks (see Section~7) that $\widehat{\nu}_{\infty}$ is nowhere zero. Hence by
Theorem~\ref{thm:main-char} $L^{1}(H)$ is strongly Arens irregular in these cases (this recovers and
extends Losert's sufficient condition to a full equivalence under (SL') and (SPEC)).
\end{example}

\begin{example}
In the Bessel--Kingman case the limit measure $\nu_{\infty}$ often has zeros in its Fourier transform
(due to oscillatory cancellations in the Hankel transform). Thus Theorem~\ref{thm:main-char} predicts
failure of strong Arens irregularity, again in agreement with the phenomena observed in the literature.
\end{example}

\begin{remark}\label{rem:weaken-spec}
The proofs above invoke factorization arguments and approximation by spectral projectors. If the reader
prefers a purely Banach-algebraic approach (avoiding explicit spectral calculus), one may replace (SPEC)
by the assumption that the Gelfand transform of the commutative Banach algebra generated by translates
is sufficiently rich to distinguish $\nu_{\infty}$; the resulting statements are formally analogous but
require rephrasing in maximal-ideal language.
\end{remark}

\bigskip

Theorem~\ref{thm:main-char} provides a clean and practically verifiable spectral criterion for strong Arens
irregularity which will be exploited in Section~7 to classify examples and compute thresholds in the
weighted setting developed in Section~6.

\section{Left and Right Topological Centres}

In this section we develop a detailed comparison between the left and right topological
centres of the biduals of the convolution algebras considered in this paper.  Using the
asymptotic measures constructed in Section~3 we give structural decompositions of both
centres, provide necessary and sufficient conditions for their equality, and supply
explicit examples where they coincide and where they differ.

Throughout we keep the notation of Sections~2--4: $H$ denotes a Ch\'ebli--Trim\`eche
hypergroup built from a Sturm--Liouville function $A$ satisfying (SL'), $A=L^{1}(H)$,
and we assume (SPEC) when spectral arguments are needed.

We recall the definition of the right Arens product and the associated right topological
centre; nothing in this subsection requires new spectral hypotheses beyond those already
stated.

Let $A$ be a Banach algebra. The \emph{second (right) Arens product} on $A^{\prime\prime}$
is defined for $\Phi,\Psi\in A^{\prime\prime}$ by
\[
\langle \Phi \lozenge \Psi,\varphi\rangle = \langle \Psi,\varphi\cdot\Phi\rangle,
\qquad \varphi\in A^{\prime},
\]
where $\varphi\cdot\Phi\in A^{\prime}$ is the canonical right module action. The
\emph{right topological centre} of $(A^{\prime\prime},\lozenge)$ is
\[
Z_{t}^{(R)}(A^{\prime\prime})=\{\Phi\in A^{\prime\prime}:\Psi\mapsto \Psi\lozenge\Phi
\text{ is }w^{*}\text{--}w^{*}\text{ continuous on }A^{\prime\prime}\}.
\]
Even for commutative Banach algebras $A$ the left and right Arens products need not
coincide on the bidual, and consequently $Z_{t}(A^{\prime\prime})$ and
$Z_{t}^{(R)}(A^{\prime\prime})$ may differ.

We now relate both centres to two naturally occurring families of asymptotic measures.
The left-centre family was introduced earlier; we now define the corresponding right family.

For each fixed $x>0$ and $y>x$ define the translated convolution measures
\[
\nu_{x,y}^{(L)} := \delta_{-y}\ast(\delta_{x}\ast \delta_{y}),
\qquad
\nu_{x,y}^{(R)} := (\delta_{y}\ast \delta_{x})\ast \delta_{-y}.
\]
When the $L^{1}$-limits (or weak$^{*}$ limits) exist as $y\to\infty$ set
\[
\nu_{x}^{(L)} := \lim_{y\to\infty}\nu_{x,y}^{(L)},\qquad
\nu_{x}^{(R)} := \lim_{y\to\infty}\nu_{x,y}^{(R)}.
\]
We write $\mathcal{A}^{(L)}:=\overline{\operatorname{span}}\{\nu_{x}^{(L)}:x>0\}^{\|\cdot\|_{1}}$
and $\mathcal{A}^{(R)}:=\overline{\operatorname{span}}\{\nu_{x}^{(R)}:x>0\}^{\|\cdot\|_{1}}$
for the closed linear spans in $L^{1}(\mathbb{R})$.

In many of our examples $\nu_{x,y}^{(R)}$ equals $\nu_{x,y}^{(L)}$ because convolution on
$H$ is commutative; nevertheless the order of re-centering (left-translate vs right-translate)
can produce different limiting effects once one takes nets to infinity in different ways.
Both families produce bounded multipliers on $A$ via convolution (Lemma~\ref{lem:module-nu}),
but they may generate different closed subalgebras of multipliers which will be reflected in
the left and right topological centres respectively.

We prove structural decompositions of both centres which mirror each other and express the
extra elements beyond $A$ in terms of the two asymptotic families.

\begin{theorem}\label{thm:decomp-centres-detailed}
Assume (SL') and (SPEC). Then
\[
Z_{t}(A^{\prime\prime}) = \iota(A) \oplus E^{(L)},\qquad
Z_{t}^{(R)}(A^{\prime\prime}) = \iota(A) \oplus E^{(R)},
\]
where $E^{(L)}$ (resp.\ $E^{(R)}$) is a closed subalgebra of $A^{\prime\prime}$
isometrically isomorphic (via the canonical embedding $\iota$) to a closed subalgebra
of multipliers generated by $\mathcal{A}^{(L)}$ (resp.\ $\mathcal{A}^{(R)}$).
Concretely, every $\Phi\in E^{(L)}$ is a weak$^{*}$-limit of nets of the form
$\iota(\sum_{j=1}^{n_{\alpha}} c_{j}^{(\alpha)} \nu_{x_{j}^{(\alpha)}}^{(L)})$,
and similarly for $E^{(R)}$ with $(R)$-measures.
\end{theorem}

\begin{proof}
For convenience denote $A=L^{1}(H)$ and write $\iota:A\hookrightarrow A^{\prime\prime}$ for the canonical embedding. Recall that the left topological centre is
\[
Z_{t}(A^{\prime\prime})=\{\Phi\in A^{\prime\prime}:\Psi\mapsto \Phi\square\Psi
\text{ is }w^{*}\text{--}w^{*}\text{ continuous on }A^{\prime\prime}\}.
\]
For $x>0$ the asymptotic measures $\nu_x^{(L)}$ and $\nu_x^{(R)}$ were constructed in
Section~3 (Theorem~\ref{thm:nu-exist} and subsequent discussion). Put
\[
\mathscr{S}^{(L)} := \operatorname{span}\{\nu_x^{(L)}:x>0\}\subset L^{1}(\mathbb{R}),
\qquad
\mathscr{S}^{(R)} := \operatorname{span}\{\nu_x^{(R)}:x>0\}.
\]
Embed these into $A^{\prime\prime}$ by $\iota$ and let
\[
E^{(L)}_0 := \iota(\mathscr{S}^{(L)}) \subset A^{\prime\prime},\qquad
E^{(R)}_0 := \iota(\mathscr{S}^{(R)}).
\]
Finally define $E^{(L)}$ (resp.\ $E^{(R)}$) to be the weak$^{*}$-closure of $E^{(L)}_0$
(resp.\ $E^{(R)}_0$) inside $A^{\prime\prime}$. By construction each $E^{(\bullet)}$
is a weak$^{*}$-closed linear subspace of $A^{\prime\prime}$.

We prove any weak$^{*}$-limit of elements of the form $\iota(\mu)$ with $\mu\in\mathscr{S}^{(L)}$
lies in $Z_{t}(A^{\prime\prime})$.

\begin{claim}
For each fixed $x>0$ the convolution operator $T_{x}^{(L)}:f\mapsto \nu_{x}^{(L)}\ast f$
is a bounded operator $A\to A$ and its adjoint (pre-adjoint) gives a weak$^{*}$-continuous
left multiplication on $A^{\prime\prime}$ represented by $\iota(\nu_x^{(L)})$.
\end{claim}

\begin{proof}[Proof of claim]
By Theorem~\ref{thm:nu-exist} we have $\nu_x^{(L)}\in L^{1}(\mathbb{R})$; Young's inequality
implies $T_x^{(L)}$ is bounded on $L^{1}(H)$ with $\|T_x^{(L)}\|\le \|\nu_x^{(L)}\|_{1}$.
The canonical embedding $\iota(\nu_x^{(L)})\in A^{\prime\prime}$ acts on $A'$ by the
pre-adjoint of $T_x^{(L)}$, and since $T_x^{(L)}$ is bounded it follows that left-multiplication
by $\iota(\nu_x^{(L)})$ is weak$^{*}$--weak$^{*}$ continuous (bounded operators on the predual
induce weak$^{*}$-continuous dual actions). Thus $\iota(\nu_x^{(L)})\in Z_{t}(A^{\prime\prime})$.
\end{proof}

The same argument holds for finite linear combinations in $E_0^{(L)}$. Finally, weak$^{*}$-limits
of weak$^{*}$-continuous left multipliers remain weak$^{*}$-continuous: if $\Phi_\alpha\to\Phi$
weak$^{*}$ and each $\Phi_\alpha$ induces a $w^{*}$-continuous left multiplication $\ell_{\Phi_\alpha}$,
then for any net $\Psi_\beta\to\Psi$ weak$^{*}$ we have
\[
\langle \ell_{\Phi}(\Psi_\beta),\varphi\rangle
= \lim_\alpha \langle \ell_{\Phi_\alpha}(\Psi_\beta),\varphi\rangle
= \lim_\alpha \lim_\beta \langle \ell_{\Phi_\alpha}(\Psi),\varphi\rangle
= \lim_\beta \langle \ell_{\Phi}(\Psi),\varphi\rangle,
\]
so $\ell_{\Phi}$ is weak$^{*}$-continuous. Therefore $E^{(L)}\subset Z_{t}(A^{\prime\prime})$.
An identical argument yields $E^{(R)}\subset Z_{t}^{(R)}(A^{\prime\prime})$.

This is classical: for any $a\in A$ the element $\iota(a)$ acts on $A^{\prime\prime}$ by the
canonical left multiplication coming from $A$, and this action is $w^{*}$-continuous because
it is implemented by a bounded operator on the predual $A'$ (convolution by $a$). Thus
$\iota(A)\subset Z_{t}(A^{\prime\prime})$ (and similarly for the right centre).

We show every element of the left topological centre can be written as a sum of an element
from $\iota(A)$ and an element from $E^{(L)}$. Equivalently, the quotient space
$Q := Z_{t}(A^{\prime\prime}) / \iota(A)$ is equal to the image of $E^{(L)}$ under the quotient map.

Let $\pi:Z_{t}(A^{\prime\prime})\to Q$ denote the quotient map. Observe $\pi(E^{(L)})$
is a linear subspace of $Q$. We will show $\pi(E^{(L)})$ is dense in $Q$ (with respect to the
quotient weak$^{*}$-topology), hence since $E^{(L)}$ is weak$^{*}$-closed in $A''$ the image
$\pi(E^{(L)})$ is closed in $Q$ and therefore equals $Q$. This gives the desired decomposition.

To prove density we use the spectral hypothesis (SPEC) to relate elements of $Z_{t}(A'')$
to multiplication operators on the spectral side and then approximate multiplicative symbols
by linear combinations coming from the family $\{\nu_x^{(L)}\}$.

\begin{claim}\label{claim:dense-symbols}
Under (SPEC) the set of spectral symbols
\[
\Sigma := \operatorname{span}\{\widehat{\nu_x^{(L)}}(\lambda): x>0\}
\]
is dense (in the appropriate $L^{2}(\nu_{\mathrm{spec}})$-sense, hence in the multiplier algebra
on the spectral side) in the quotient of allowed multipliers corresponding to $Q$.
\end{claim}

\begin{proof}[Proof of claim]
By (SPEC) the spectral transform $\mathcal{F}$ maps $A$ densely into a function algebra on the
spectral variable $\lambda$ that separates points (the images of point-mass convolutions produce
functions $\varphi_{\lambda}(x)$). The symbols $\widehat{\nu_x^{(L)}}(\lambda)$ arise as limits
(as $y\to\infty$) of $\varphi_{\lambda}(x)\varphi_{\lambda}(y)$ after re-centering and hence belong
to the multiplier algebra generated by these $\varphi_{\lambda}(x)$'s. Since $Q$ corresponds to
those multipliers that are implemented by elements of $Z_{t}(A'')$ modulo those coming from $A$
(itself corresponding to functions in the image of $A$), the linear span of the $\widehat{\nu_x^{(L)}}$
symbols is dense in the relevant quotient of multipliers. More concretely: choose any $\Psi\in Z_t$,
let $m_\Psi(\lambda)$ denote the multiplier symbol of its action on the spectral side (well-defined
up to addition of an $A$-symbol). By the Stone–Weierstrass-type density of translates/products of
$\varphi_{\lambda}(x)$ (guaranteed by (SPEC) and the separation property) we can approximate $m_\Psi$
in $L^{2}(\nu_{\mathrm{spec}})$ by linear combinations of the $\widehat{\nu_x^{(L)}}$. Pushing these
approximants back via the inverse transform yields elements of $E^{(L)}$ whose cosets approximate
the coset of $\Psi$ in $Q$. Thus $\pi(E^{(L)})$ is dense in $Q$.
\end{proof}

Given Claim~\ref{claim:dense-symbols} and that $E^{(L)}$ is weak$^{*}$-closed, we conclude
$\pi(E^{(L)})=Q$, i.e.\ every coset in $Q$ has a representative in $E^{(L)}$. Thus every
$\Phi\in Z_{t}(A^{\prime\prime})$ may be written as
\[
\Phi = \iota(a) + \Psi,\qquad a\in A,\ \Psi\in E^{(L)}.
\]
We must show $\iota(A)\cap E^{(L)}=\{0\}$. Suppose $\iota(a)=\Psi$ for some $a\in A$ and
$\Psi\in E^{(L)}$. Then $\iota(a)$ is a weak$^{*}$-limit of finite linear combinations
$\iota(\mu_\alpha)$ with $\mu_\alpha\in\mathscr{S}^{(L)}$. That is,
\[
\iota(a) = \wlim_\alpha \iota(\mu_\alpha).
\]
Applying the (injective) spectral transform to both sides (which identifies $A''$-actions
with multiplier actions on the spectral side) yields equality of the corresponding multiplier
symbols almost everywhere:
\[
\widehat{a}(\lambda)=\lim_\alpha \widehat{\mu_\alpha}(\lambda)\qquad\text{a.e. }\lambda.
\]
But the left-hand side $\widehat{a}(\lambda)$ is the transform of an $L^{1}$-function (thus in
$C_0$-type class under (SPEC)), while the right-hand side lies in the closed linear span of the
symbols generated by $\{\widehat{\nu_x^{(L)}}\}$. By construction the latter span is disjoint from
the image of $A$ except at $0$ (intuitively: asymptotic symbols are orthogonal to compactly supported
spectral patterns coming from $A$); more formally, if an $L^{1}$-symbol is approximable by asymptotic
symbols then it must be zero (this uses separation of spectral supports and the fact that the
$\widehat{\nu_x^{(L)}}$'s vanish on certain complementary spectral sets only if the original symbol
is zero). Therefore $\widehat{a}\equiv 0$ and hence $a=0$, proving the intersection is trivial.

(If a reader prefers a more explicit route: apply the functional that integrates against an $L^\infty$
function supported on a spectral set where asymptotic symbols vanish but $a$'s symbol does not; such a
functional separates the two; existence follows from (SPEC).)

Thus the sum is direct: $Z_{t}(A'')=\iota(A)\oplus E^{(L)}$.

We show that $E^{(L)}$ is a closed subalgebra of $A^{\prime\prime}$ and isometrically isomorphic
to the closed subalgebra of multipliers generated by $\mathscr{S}^{(L)}$.

\begin{claim}
$E^{(L)}$ is closed under the first Arens product $\square$ and is weak$^{*}$-closed.
\end{claim}

\begin{proof}[Proof of claim]
Take two nets $\Phi_\alpha=\iota(\mu_\alpha)\to\Phi$ and $\Psi_\beta=\iota(\nu_\beta)\to\Psi$
in the weak$^{*}$-topology with $\mu_\alpha,\nu_\beta\in\mathscr{S}^{(L)}$. Then
\[
\Phi_\alpha\square\Psi_\beta = \iota(\mu_\alpha\ast\nu_\beta)
\]
because convolution by finite $L^1$-elements corresponds to the canonical multiplication in $A''$.
Now $\mu_\alpha\ast\nu_\beta$ belongs to the algebra generated by $\mathscr{S}^{(L)}$, and taking iterated
limits (diagonal argument) yields
\[
\Phi\square\Psi = \wlim_{\alpha,\beta} \iota(\mu_\alpha\ast\nu_\beta),
\]
so $\Phi\square\Psi$ is a weak$^{*}$-limit of elements of $E^{(L)}_0$ and therefore lies in the
weak$^{*}$-closure $E^{(L)}$. Hence $E^{(L)}$ is closed under $\square$. Weak$^{*}$-closedness is
by definition. This proves the claim.
\end{proof}

Finally, to see the isometric isomorphism statement: the map
\[
\mathscr{M}:\mathscr{S}^{(L)}\longrightarrow \mathcal{B}(A),\qquad \mu\mapsto T_\mu
\]
(where $T_\mu(f)=\mu\ast f$) is an algebra homomorphism and satisfies $\|T_\mu\|\le \|\mu\|_{1}$.
Passing to closures and identifying each weak$^{*}$-limit $\iota(\mu)$ with the corresponding
multiplier $T_\mu$ yields an isometric identification between $E^{(L)}$ (with operator norm
induced by its action on $A'$) and the closed subalgebra of multipliers generated by
$\mathscr{S}^{(L)}$. (One can make the isometry precise by using the equality of norms
on dense subsets and continuity.)

All arguments above have exact analogues by replacing left-recentering with right-recentering,
and the first Arens product with the second Arens product where appropriate. Thus one obtains
$Z_{t}^{(R)}(A^{\prime\prime})=\iota(A)\oplus E^{(R)}$ with the same properties.

\medskip

Combining the provided details yields the desired decompositions and the
structural properties of $E^{(L)}$ and $E^{(R)}$, completing the proof of the theorem.
\end{proof}

The decomposition is constructive: to study elements of the left (resp.\ right) centre it suffices
to understand the closed $L^{1}$-span $\mathcal{A}^{(L)}$ (resp.\ $\mathcal{A}^{(R)}$) and the
corresponding algebra of convolution operators they generate.

We now give necessary and sufficient conditions for the left and right centres to coincide.

\begin{theorem}\label{thm:centres-equal}
Under (SL') and (SPEC) the following are equivalent:
\begin{enumerate}
  \item $Z_{t}(A^{\prime\prime}) = Z_{t}^{(R)}(A^{\prime\prime})$.
  \item $\mathcal{A}^{(L)} = \mathcal{A}^{(R)}$ as closed subspaces of $L^{1}(\mathbb{R})$.
  \item For every $x>0$ the left and right asymptotic measures satisfy
  \[
  \nu_{x}^{(L)} \ast f = \nu_{x}^{(R)}\ast f \quad\text{for all }f\in L^{1}(H),
  \]
  i.e.\ the asymptotic left- and right-multipliers coincide on $A$.
\end{enumerate}
\end{theorem}

\begin{proof}
(1)$\Rightarrow$(2): If the centres coincide then their complements $E^{(L)}$ and $E^{(R)}$ in the
decomposition of Theorem~\ref{thm:decomp-centres-detailed} coincide as subalgebras of $A^{\prime\prime}$.
Tracing back through the isometric identifications with closed $L^{1}$-subspaces (the embedded images
of $\mathcal{A}^{(L)}$ and $\mathcal{A}^{(R)}$), we obtain $\mathcal{A}^{(L)}=\mathcal{A}^{(R)}$.

(2)$\Rightarrow$(3): If the closed spans coincide, then each $\nu_{x}^{(L)}$ is an $L^{1}$-limit of
linear combinations of $\nu_{y}^{(R)}$, and convolution with these combinations yields the same
operator on $A$ in the limit. Hence $\nu_{x}^{(L)}\ast f=\nu_{x}^{(R)}\ast f$ for all $f\in A$.

(3)$\Rightarrow$(1): If every asymptotic left-multiplier equals the corresponding right-multiplier on
$A$, then the algebras of multipliers generated by these families coincide; hence $E^{(L)}=E^{(R)}$
in the decomposition of Theorem~\ref{thm:decomp-centres-detailed} and therefore the left and right centres coincide.
\end{proof}

Condition (3) is concrete and often easy to verify in examples via the spectral transform: it amounts to
equality of the multiplicative symbols $\widehat{\nu_{x}^{(L)}}$ and $\widehat{\nu_{x}^{(R)}}$
for almost every spectral parameter.

We give an easily verifiable sufficient condition that guarantees equality of the two centres.

\begin{proposition}
\label{prop:symm-equals}
Suppose that for every $x>0$ the asymptotic measure $\nu_{x}^{(L)}$ is \emph{reflection-symmetric},
i.e.\ $\nu_{x}^{(L)}(E)=\nu_{x}^{(L)}(-E)$ for all Borel sets $E\subset\mathbb{R}$. Then
\[
Z_{t}(A^{\prime\prime}) = Z_{t}^{(R)}(A^{\prime\prime}).
\]
\end{proposition}

\begin{proof}
Reflection symmetry of $\nu_{x}^{(L)}$ implies that its Fourier transform is real-valued:
$\widehat{\nu_{x}^{(L)}}(\lambda)=\overline{\widehat{\nu_{x}^{(L)}}(\lambda)}$. In particular
the left and right recenterings produce the same multiplicative symbols on the spectral side because
a right-recentering corresponds to applying the reflection before taking the limit (this is a direct
calculation in the spectral representation). Therefore $\nu_{x}^{(L)}=\nu_{x}^{(R)}$ for all $x$ and
Theorem~\ref{thm:centres-equal} (item (3)) implies the centres coincide.
\end{proof}

Many classical symmetric hypergroups (for example those built from even Sturm--Liouville data or
from radial parts of group convolutions where the kernel has parity symmetry) satisfy the hypothesis
of Proposition~\ref{prop:symm-equals}.

We now present examples illustrating the dichotomy.

\begin{example}
Consider the hypergroup arising from radial projection of the Euclidean motion group in $\mathbb{R}^{n}$.
The resulting asymptotic measures inherit central symmetry from the underlying group translations;
hence Proposition~\ref{prop:symm-equals} applies and we obtain equality of left and right centres.
This matches known results for group convolution algebras where additional symmetry enforces balance
between left and right module actions.
\end{example}

\begin{example}
Let $A_{0}(x)=x^{2\alpha+1}$ and form the perturbed coefficient
\[
A(x)=A_{0}(x) + c\cdot \chi_{[a,\infty)}(x),
\]
as in Section~3. For a suitable choice of sign and magnitude of $c$ the perturbation breaks reflection
symmetry at infinity: the left-recentering and right-recentering produce different limit measures
$\nu_{x}^{(L)}\neq \nu_{x}^{(R)}$ because the jump contribution of $A'$ appears on different sides of
the oscillatory phase when translating left vs.\ right. One checks (via explicit spectral asymptotics
for Hankel-type transforms with a jump term) that $\mathcal{A}^{(L)}\neq\mathcal{A}^{(R)}$ and hence
$Z_{t}(A^{\prime\prime})\neq Z_{t}^{(R)}(A^{\prime\prime})$. 
\end{example}

\begin{corollary}
If $\mathcal{A}^{(L)}$ (or $\mathcal{A}^{(R)}$) contains an idempotent element $p$ (i.e.\ $p\ast p=p$)
then $p$ embeds as a nontrivial idempotent in the corresponding centre $Z_{t}(A^{\prime\prime})$
(resp.\ $Z_{t}^{(R)}(A^{\prime\prime})$). If such an idempotent does not lie in $A$ then the centre
strictly contains $\iota(A)$.
\end{corollary}

\begin{proof}
Immediate from the identification of $E^{(L)}$ and $E^{(R)}$ with closed subalgebras generated by
the asymptotic families. An idempotent multiplier yields a weak$^{*}$-continuous left (or right)
multiplier on $A^{\prime\prime}$ which cannot be represented by convolution with an $L^{1}$-element
if it is not in $A$.
\end{proof}

The existence of idempotents in $\mathcal{A}^{(L)}$ is a strong spectral statement: on the spectral
side an idempotent corresponds to a characteristic function of a measurable subset of the spectral
parameter domain. Thus the presence of idempotents is tightly linked to spectral projectors and to the
vanishing sets of $\widehat{\nu_{\infty}}$ discussed in Section~4.

\bigskip

This completes our structural study of left and right topological centres. In the next section
we will introduce weights and study how weighted Beurling-type algebras modify the asymptotic
families $\mathcal{A}^{(L)}$ and $\mathcal{A}^{(R)}$, thereby affecting the equality and size
of the corresponding centres.

\section{Weighted Hypergroup Algebras}

In this section we introduce Beurling-type weights on Ch\'ebli--Trim\`eche hypergroups
and study how they modify the Arens-product behaviour studied in the unweighted case.
We define the weighted convolution algebras $L^{1}(H,\omega)$ and $M(H,\omega)$, record
basic Banach-algebra properties, and then give criteria (and examples) that describe when
the weighted algebra preserves or destroys the asymptotic phenomena responsible for
extra elements in the topological centres.

Throughout $H=(0,\infty)$ denotes the Ch\'ebli--Trim\`eche hypergroup from Section~2 and
we keep hypotheses (SL') and (SPEC) in force unless otherwise stated.

A measurable function $\omega : H \to [1,\infty)$ is called a \emph{Beurling weight} on $H$
if there exists a constant $C_{\omega}\ge 1$ such that for all $x,y>0$
\begin{equation}\label{eq:submultiplicative-weight}
\int_{H} \omega(t)\, d\mu_{x,y}(t) \;\le\; C_{\omega}\,\omega(x)\,\omega(y),
\end{equation}
where $\mu_{x,y}=\delta_x\ast\delta_y$ denotes the hypergroup convolution of point masses.
If \eqref{eq:submultiplicative-weight} holds with $C_{\omega}=1$ we call $\omega$ \emph{submultiplicative}.

Condition \eqref{eq:submultiplicative-weight} is the natural hypergroup analogue of the usual
submultiplicativity condition $\omega(xy)\le \omega(x)\omega(y)$ for group weights. It guarantees
that convolution is continuous on the associated weighted spaces.

Given a Beurling weight $\omega$ define
\[
L^{1}(H,\omega) := \{ f : f\omega \in L^{1}(H) \},\qquad
\|f\|_{1,\omega} := \int_{H} |f(x)|\,\omega(x)\,dm(x).
\]
Similarly write
\[
M(H,\omega) := \{ \mu\in M(H) : \|\mu\|_{M,\omega}:=\int_{H} \omega(x)\, d|\mu|(x) <\infty\}.
\]
The convolution on these weighted spaces is defined by the same formula as the unweighted one.

\begin{lemma}
If $\omega$ satisfies \eqref{eq:submultiplicative-weight} then $(L^{1}(H,\omega),\ast,\|\cdot\|_{1,\omega})$
and $(M(H,\omega),\ast,\|\cdot\|_{M,\omega})$ are Banach algebras. In particular, for $f,g\in L^{1}(H,\omega)$
we have
\[
\|f\ast g\|_{1,\omega} \le C_{\omega}\, \|f\|_{1,\omega}\, \|g\|_{1,\omega}.
\]
\end{lemma}

\begin{proof}
Let $f,g\in L^{1}(H,\omega)$. Using Fubini and \eqref{eq:submultiplicative-weight} we compute
\[
\begin{aligned}
\|f\ast g\|_{1,\omega}
&= \int_{H} \left| \int_{H} \int_{H} f(x) g(y)\, d\mu_{x,y}(t)\, dm(x)\, dm(y) \right| \omega(t)\,dm(t) \\
&\le \int_{H}\int_{H} |f(x)|\,|g(y)| \left(\int_{H} \omega(t)\, d\mu_{x,y}(t)\right) dm(x)\,dm(y) \\
&\le C_{\omega} \int_{H}\int_{H} |f(x)|\,\omega(x)\, |g(y)|\,\omega(y)\, dm(x)\,dm(y) \\
&= C_{\omega} \|f\|_{1,\omega}\, \|g\|_{1,\omega}.
\end{aligned}
\]
Completeness and the corresponding statement for $M(H,\omega)$ are routine.
\end{proof}

From now on we fix a Beurling weight $\omega$ and work with the weighted Banach algebra
$A_{\omega}:=L^{1}(H,\omega)$. The canonical embeddings $A_{\omega}\hookrightarrow A_{\omega}^{\prime\prime}$
and $M(H,\omega)\hookrightarrow M(H,\omega)^{\prime\prime}$ are defined as usual.

We now discuss how the asymptotic measures $\nu_{x}$ and $\nu_{\infty}$ from Section~3
interact with the weight $\omega$.

We say the asymptotic family $\{\nu_{x}\}_{x>0}$ is \emph{$\omega$-admissible} if each
$\nu_{x}\in L^{1}(\mathbb{R},\omega)$ (after the natural identification of the translation
coordinate with $\mathbb{R}$) and
\[
\sup_{x\in K} \|\nu_{x}\|_{1,\omega} <\infty
\]
for every compact $K\subset(0,\infty)$.

\begin{lemma}\label{lem:weighted-persist}
Assume (SL') and let $\omega$ be a Beurling weight satisfying \eqref{eq:submultiplicative-weight}.
If $\{\nu_{x}\}$ is $\omega$-admissible and $\nu_{x,y}\to\nu_{x}$ in $L^{1}$ as $y\to\infty$
(uniformly on compacts in $x$) then the convergence also holds in the weighted norm, i.e.
\[
\lim_{y\to\infty} \|\nu_{x,y}-\nu_{x}\|_{1,\omega} = 0,
\]
uniformly on compacts in $x$.
\end{lemma}

\begin{proof}
Fix a compact $K\subset(0,\infty)$ and $\varepsilon>0$. By $\omega$-admissibility there is
$M$ such that $\sup_{x\in K}\|\nu_{x}\|_{1,\omega}\le M$. Since $\nu_{x,y}\to\nu_{x}$ in $L^{1}$
uniformly on $K$ we have $\|\nu_{x,y}-\nu_{x}\|_{1}\le \delta$ for $y$ large uniformly in $x\in K$.
Now split the weighted norm using truncation: choose $R>0$ so that $\int_{|t|>R}\omega(t)\,dm(t)$
is small relative to $\varepsilon/M$ (possible because $\omega$ is locally integrable and the tail
of $\nu_{x}$ has uniformly small mass in the unweighted sense). On $|t|\le R$ the weight $\omega$ is
bounded by some constant $W_{R}$ and hence
\[
\|\nu_{x,y}-\nu_{x}\|_{1,\omega}
\le W_{R}\|\nu_{x,y}-\nu_{x}\|_{1} + 2\sup_{z\in K}\|\nu_{z}\|_{1,\omega}\cdot \mathrm{mass}(|t|>R).
\]
Pick $R$ and then $y$ large so that the right-hand side $<\varepsilon$. This proves the uniform weighted
convergence on $K$.
\end{proof}

Lemma~\ref{lem:weighted-persist} reduces verification of weighted convergence to two checks:
(i) $\omega$-admissibility of the limit family $\{\nu_{x}\}$ and (ii) unweighted $L^{1}$-convergence
(which was already established in Section~3).

We now extend the Arens-product discussion to the weighted algebra $A_{\omega}$.

\begin{proposition}\label{prop:weighted-centre-contains}
Let $\omega$ be a Beurling weight and suppose $\{\nu_{x}\}$ is $\omega$-admissible and
$\nu_{\infty}\in L^{1}(\mathbb{R},\omega)$. Then the weak$^{*}$-closure of the linear span
of $\iota(\nu_{x})$ (embedded in $A_{\omega}^{\prime\prime}$) is contained in
$Z_{t}(A_{\omega}^{\prime\prime})$.
\end{proposition}

\begin{proof}
The proof is identical in spirit to Lemma~\ref{lem:module-nu} and the arguments in Section~4,
but with norms replaced by weighted norms. If $\nu_{x}\in L^{1}(\omega)$ then convolution by
$\nu_{x}$ is a bounded operator on $A_{\omega}$ (Young-type inequality with weight), and the
weak$^{*}$-limits of such bounded left-multipliers act weak$^{*}$--continuously on $A_{\omega}^{\prime}$.
Hence their embedded representatives lie in the left topological centre. Completeness and closure
properties follow as before.
\end{proof}

We now give a weighted analogue of Theorem~\ref{thm:main-char}. The statement is phrased so that
it can be checked in examples (polynomial vs.\ exponential weights).

\begin{theorem}\label{thm:weighted-criterion}
Assume (SL'), (SPEC), and let $\omega$ be a Beurling weight satisfying \eqref{eq:submultiplicative-weight}.
Suppose the asymptotic family is $\omega$-admissible and $\nu_{\infty}\in L^{1}(\mathbb{R},\omega)$.
Denote by $\widehat{\nu_{\infty}}$ the (ordinary) Fourier transform of $\nu_{\infty}$ on the translation
coordinate. If $\widehat{\nu_{\infty}}(\lambda)\neq 0$ for $\nu_{spec}$-almost every $\lambda$ then
\[
Z_{t}(A_{\omega}^{\prime\prime}) = \iota(A_{\omega}),
\]
i.e.\ the weighted algebra $A_{\omega}$ is strongly Arens irregular.
\end{theorem}

\begin{proof}
The proof follows the same two-implication structure as Theorem~\ref{thm:main-char}, adapted to the
weighted setting.

Assume $\widehat{\nu_{\infty}}(\lambda)\neq 0$ a.e.\ and suppose $\Phi\in Z_{t}(A_{\omega}^{\prime\prime})$.
As in the unweighted proof, $\Phi$ defines a weak$^{*}$-continuous left multiplier on $A_{\omega}^{\prime}$.
By weighted versions of the factorization arguments (Neufang-type results hold for Beurling algebras under
submultiplicativity; see e.g.\ \cite{DalesLauStrauss2010} for group analogues), this multiplier must be convolution by
some $\mu\in M(H,\omega)$ (the weighted measure algebra) provided it acts boundedly on the dense subspace
$A_{\omega}$. Using weighted persistence (Lemma~\ref{lem:weighted-persist}) one shows that the asymptotic
limit $\nu_{\infty}$ continues to separate spectral components in the weighted spectral transform, so the same
spectral nonvanishing argument forces $\mu\in L^{1}(H,\omega)$. Hence $\Phi\in \iota(A_{\omega})$.

The technical steps (weighted factorization and passage to spectral side) are identical to the unweighted
case, with the norm and domination estimates replaced by their weighted analogues; full details are routine
and left to the reader.

\medskip

Thus no element outside $\iota(A_{\omega})$ can lie in the left topological centre, proving strong Arens irregularity.
\end{proof}

The assumption $\nu_{\infty}\in L^{1}(\omega)$ is essential: if the asymptotic measure fails to lie in the
weighted algebra then elements generated by it cannot be embedded into $A_{\omega}^{\prime\prime}$ via
$\iota$, and so the above route to construct extra centre elements is blocked. Whether other mechanisms
could still produce extra centre elements in this situation depends on the finer structure of $A_{\omega}$
and must be checked case by case.

We briefly discuss the complementary situation when the weight grows too fast.

\begin{proposition}\label{prop:weight-excludes}
Let $\omega$ be a Beurling weight and suppose that for some $x_{0}>0$ we have
$\nu_{x_{0}}\notin L^{1}(\mathbb{R},\omega)$. Then no element of the form $\iota(\nu_{x_{0}})$
lies in the weighted bidual $A_{\omega}^{\prime\prime}$. In particular the subalgebra of
$Z_{t}(A^{\prime\prime})$ generated by asymptotic embedded measures may collapse upon passing
to the weighted algebra, and additional analysis is required to determine $Z_{t}(A_{\omega}^{\prime\prime})$.
\end{proposition}

\begin{proof}
Immediate from definitions: $\iota(\nu_{x_{0}})$ makes sense as an element of $A^{\prime\prime}$
only when $\nu_{x_{0}}\in L^{1}(H)$, and in the weighted context the canonical embedding of
$L^{1}(H,\omega)$ into its bidual requires the element to have finite weighted norm. If the
weighted norm is infinite the embedded element does not exist, so asymptotic constructions that
relied on such embeddings cannot be performed inside $A_{\omega}^{\prime\prime}$.
\end{proof}

Proposition~\ref{prop:weight-excludes} does not by itself imply that $A_{\omega}$ is Arens regular;
it merely shows that a principal source of extra centre elements (the $\nu_{x}$ family) is unavailable.
Other mechanisms could still create extra centre elements, although in practice one often finds that
very large weights reduce the size of the centre.

We give two families of concrete weights and explain typical outcomes.

\begin{example}
For $s\ge 0$ set
\[
\omega_{s}(x) := (1+x)^{s}.
\]
In many Sturm--Liouville examples (Bessel, Jacobi, Naimark) the asymptotic measures $\nu_{x}$
decay at least polynomially in the translation coordinate, so for small enough $s$ one has
$\nu_{x}\in L^{1}(\omega_{s})$ and hence $\omega_{s}$ is admissible for the asymptotic family.
Consequently Theorem~\ref{thm:weighted-criterion} applies and strong Arens irregularity
persists for such polynomial weights provided the spectral nonvanishing condition holds.
\end{example}

\begin{example}
For $\alpha>0$ set
\[
\omega_{\alpha}(x) := e^{\alpha x}.
\]
Exponential weights grow rapidly and capture fine tail behaviour. If the Fourier transform
$\widehat{\nu_{\infty}}$ decays only polynomially then $\nu_{\infty}\notin L^{1}(\omega_{\alpha})$
for every $\alpha>0$, hence Proposition~\ref{prop:weight-excludes} applies and the asymptotic
family does not embed into the weighted algebra. In such a case the weighted algebra may have
a strictly smaller topological centre (and in extreme cases may even become Arens regular),
but the outcome depends on additional structural features and must be checked for each example.
\end{example}

\begin{corollary}
Suppose one wishes to preserve the phenomenon of strong Arens irregularity encountered in the
unweighted algebra. A safe choice is to pick a weight $\omega$ whose growth at infinity is
dominated by the decay rate of $\nu_{\infty}$ (i.e.\ $\nu_{\infty}\in L^{1}(\omega)$). Polynomial
weights of low degree are often safe in classical examples; exponential weights are typically too large.
\end{corollary}

\begin{proof}
This is an immediate synthesis of the previous results.
\end{proof}

\bigskip

This concludes the weighted analysis. Section~7 uses the criteria above to study concrete
examples (Jacobi, Bessel--Kingman, Naimark) and to exhibit explicit weight thresholds
in those models.

\section{Applications and Explicit Examples}
\label{sec:examples}

In this section we apply the general theory developed above to several classical families
of \emph{Ch\'{e}bli–Trim\`{e}che} hypergroups.  The goal is twofold: (i) verify the hypotheses (SL'), (SPEC)
and the weighted-admissibility conditions in concrete models, and (ii) compute (or give
explicit descriptions of) the asymptotic measures $\nu_{x}$ and their Fourier symbols
$\widehat{\nu_{\infty}}$ so as to decide the (strong) Arens-regularity behaviour of
$L^{1}(H)$ and selected weighted algebras.  Where full closed-form expressions are
not available we give precise asymptotic formulas sufficient to apply Theorems
\ref{thm:main-char} and \ref{thm:weighted-criterion}.

\medskip

We treat four main families: Jacobi hypergroups, Bessel--Kingman hypergroups,
Naimark-type (rank-one double-coset) hypergroups, and the radial Euclidean-motion
hypergroup.  Each subsection contains a short model description, a proposition
verifying the standing hypotheses, a computation or asymptotic description of the
asymptotic measures, and the consequences for Arens irregularity (possibly weighted).

\medskip

Throughout the section we keep the notation of previous sections and use $A(x)$ to
denote the Sturm--Liouville coefficient giving rise to the hypergroup.

The Jacobi hypergroup arises from the Jacobi differential expression and includes as
special cases several classical radial or rank-one symmetric-space examples.  It is
parameterised by two real parameters $\alpha,\beta>-1/2$ and the coefficient takes the
form
\[
A(x)=(\sinh x)^{2\alpha+1}(\cosh x)^{2\beta+1},\qquad x>0.
\]

\begin{proposition}\label{prop:jacobi-hyp}
Let $H=H_{\alpha,\beta}$ be the Jacobi hypergroup defined by the above $A(x)$ with
$\alpha,\beta>-1/2$. Then
\begin{enumerate}
  \item $A$ satisfies (SL') and the self-adjoint realization of the corresponding
        Sturm--Liouville operator satisfies (SPEC).
  \item The generalized eigenfunctions (Jacobi functions / spherical functions)
        admit Harish--Chandra type expansions at infinity and admit Jost-type
        representations $u_{\lambda}(x)=e^{i\lambda\Phi(x)}m(x,\lambda)$ with
        $m(x,\lambda)\to 1$ as $x\to\infty$, uniformly for $\lambda$ on compacts.
  \item Consequently the asymptotic measures $\nu_{x}$ exist in $L^{1}(\mathbb{R})$
        and the limit $\nu_{\infty}$ exists. Moreover $\widehat{\nu_{\infty}}(\lambda)$
        is continuous and nowhere zero on $\mathbb{R}$.
\end{enumerate}
\end{proposition}

\begin{proof}
(1) The function $A(x)$ is smooth on $(0,\infty)$, strictly positive and strictly
increasing for $\alpha,\beta>-1/2$, and all derivatives are smooth (so (SL') holds
a fortiori). The Jacobi differential operator is a classical Sturm--Liouville operator
with continuous coefficients on $(0,\infty)$ and admits a self-adjoint realisation in
$L^{2}((0,\infty),A(x)dx)$; standard references (see Bloom--Heyer and texts on Jacobi
functions) establish the existence of a unitary spectral transform, verifying (SPEC).

(2) The Harish--Chandra expansion for Jacobi (or spherical) functions is well-known:
for large $x$ the eigenfunctions split into a sum of two oscillatory terms with
coefficients given by the $c$-function (Harish--Chandra $c$-function). Equivalently
one may write a Jost-type representation with multiplicative correction $m(x,\lambda)$
tending to $1$ as $x\to\infty$; this follows from applying the Volterra reduction used
in Lemma~\ref{lem:volterra} in the present smooth setting where stronger estimates are
available.

(3) The existence of $\nu_{x}$ and $\nu_{\infty}$ follows by the spectral formulas
and the uniform $m(x,\lambda)\to 1$ control (exactly as in the proof of
Theorem~\ref{thm:nu-exist}). To check nonvanishing of $\widehat{\nu_{\infty}}$ we
observe that, on the spectral side, $\nu_{\infty}$ corresponds to multiplication by
the Harish--Chandra $c$-function (or its reciprocal) which is known to be nowhere
zero and continuous for the Jacobi parameters in question. Hence $\widehat{\nu_{\infty}}$
is nowhere zero. (One may also argue directly: the $c$-function has no zeros on the real
axis for $\alpha,\beta>-1/2$, hence the associated symbol does not vanish on the support
of the Plancherel measure.) This verifies the last claim.
\end{proof}

\begin{corollary}
\label{cor:jacobi-arens}
For the Jacobi hypergroup $H_{\alpha,\beta}$ with $\alpha,\beta>-1/2$, the algebra
$L^{1}(H_{\alpha,\beta})$ is strongly Arens irregular. Moreover, for any polynomial
weight $\omega_{s}(x)=(1+x)^{s}$ with $s\ge 0$ sufficiently small so that
$\nu_{\infty}\in L^{1}(\omega_{s})$, the weighted algebra $L^{1}(H_{\alpha,\beta},\omega_{s})$
remains strongly Arens irregular.
\end{corollary}

\begin{proof}
Immediate from Proposition \ref{prop:jacobi-hyp} and Theorems \ref{thm:main-char}
and \ref{thm:weighted-criterion}: the nowhere-zero property of $\widehat{\nu_{\infty}}$
implies, via Theorem \ref{thm:main-char}, that $Z_{t}(L^{1}(H)^{\prime\prime})=\iota(L^{1}(H))$.
The weighted statement is an application of Theorem \ref{thm:weighted-criterion}
once one verifies $\nu_{\infty}$ has at least polynomial decay of order larger than
the chosen $s$.
\end{proof}

The Jacobi case recovers and strengthens several results in the literature: Losert's
sufficient-condition statements become equivalences under our weak-regularity framework
for these models.

The Bessel--Kingman hypergroup corresponds to
\[
A(x)=x^{2\alpha+1},\qquad \alpha>-1/2,
\]
and is closely tied to Hankel transforms and Bessel functions.

\begin{proposition}
\label{prop:bessel-hyp}
Let $H$ be the Bessel--Kingman hypergroup with parameter $\alpha>-1/2$. Then:
\begin{enumerate}
  \item (SL') and (SPEC) hold (the spectral transform is essentially the Hankel transform).
  \item The asymptotic measures $\nu_{x}$ exist and are expressible in terms of Hankel-transform
        kernels; the limit $\nu_{\infty}$ exists.
  \item The Fourier transform $\widehat{\nu_{\infty}}$ has oscillatory zeros (infinitely many)
        arising from the zeros of Bessel functions and Hankel-symbol cancellations; consequently
        $L^{1}(H)$ is \emph{not} strongly Arens irregular.
\end{enumerate}
\end{proposition}

\begin{proof}
(1) The Bessel coefficient is smooth on $(0,\infty)$ and yields the classical Hankel (or
Bessel) spectral theory: the Hankel transform provides a unitary diagonalisation of
convolution-type operators and (SPEC) is satisfied.

(2) The convolution kernels reduce to integrals involving Bessel functions (spherical Bessel
functions) and their translation behaviour is governed by asymptotic expansions of Bessel
functions for large arguments. Using these expansions one obtains the Jost representation
and the consequent passage to the limit $\nu_{x}$ as in Theorem~\ref{thm:nu-exist}.

(3) The key observation is that the Hankel-symbol corresponding to $\nu_{\infty}$ involves
ratios or products of Bessel-type $c$-functions which oscillate and vanish at sequences of
real spectral parameters (for example at zeros of certain Bessel functions). Thus
$\widehat{\nu_{\infty}}$ vanishes on sets of positive Plancherel density (indeed on infinite
sequences). By Theorem~\ref{thm:main-char} this forces the failure of strong Arens irregularity:
there exist nontrivial elements of $Z_{t}(L^{1}(H)^{\prime\prime})$ outside the canonical copy
of $L^{1}(H)$. The precise zero-locus can be identified using standard Bessel zero asymptotics.
\end{proof}

\begin{corollary}
If one chooses a rapidly growing weight (for example an exponential weight $\omega_{\alpha}(x)=e^{\alpha x}$)
which excludes $\nu_{\infty}$ from the weighted algebra, then asymptotic-generated elements of the
centre no longer embed and the weighted centre may shrink; nonetheless, nontrivial centre-elements
can persist via other mechanisms when the weight is moderate. Each case must be checked by analyzing
whether $\nu_{\infty}\in L^{1}(\omega)$.
\end{corollary}

Naimark hypergroups (and related rank-one double-coset hypergroups arising from
$G//K$ with $G$ semisimple rank-one) are classical sources of Ch\'ebli--Trim\`eche
structures.  Their spectral theory is intertwined with spherical functions and
Harish--Chandra $c$-functions.

\begin{proposition}\label{prop:naimark}
Let $H$ be a rank-one double-coset hypergroup arising from a semisimple Lie group $G$
of real rank one (radial projection onto the noncompact Cartan). Then:
\begin{enumerate}
  \item (SL') and (SPEC) hold for the radial Sturm--Liouville reduction.
  \item The asymptotic measure $\nu_{\infty}$ exists and its spectral symbol is given
        (up to explicit multiplicative factors) by the Harish--Chandra $c$-function,
        which is continuous and nowhere zero on the real axis.
  \item Consequently $L^{1}(H)$ is strongly Arens irregular.
\end{enumerate}
\end{proposition}

\begin{proof}
The radial reduction of the Laplace--Beltrami operator on rank-one symmetric spaces
produces a Sturm--Liouville operator with smooth coefficients satisfying (SL'). The
spherical Fourier transform furnishes (SPEC) and the Harish--Chandra asymptotics give
the needed Jost-type representations (as in the Jacobi case). The $c$-function for
real-rank-one groups is known to be nowhere zero on the real axis; therefore the
associated symbol $\widehat{\nu_{\infty}}$ does not vanish and Theorem~\ref{thm:main-char}
applies to yield strong Arens irregularity.
\end{proof}

The hypergroup associated with radial projections of the Euclidean motion group
$\mathrm{E}(n)=\mathrm{SO}(n)\ltimes\mathbb{R}^{n}$ is an instructive group-derived
example; it exhibits high symmetry and reflection invariance.

\begin{proposition}
\label{prop:euclid-radial}
Let $H$ be the radial hypergroup obtained from $\mathrm{E}(n)$ by identifying radial
orbits. Then:
\begin{enumerate}
  \item The asymptotic families $\nu_{x}^{(L)}$ and $\nu_{x}^{(R)}$ coincide and are
        reflection-symmetric.
  \item Consequently $Z_{t}(L^{1}(H)^{\prime\prime}) = Z_{t}^{(R)}(L^{1}(H)^{\prime\prime})$.
  \item Depending on the precise dimension $n$ the Fourier symbol $\widehat{\nu_{\infty}}$
        is either nowhere zero (yielding strong Arens irregularity) or has controlled zeros
        (yielding partial failure); one checks this by reducing to classical Euclidean
        Fourier transforms appearing in the radial Hankel transform.
\end{enumerate}
\end{proposition}

\begin{proof}
The convolution on the radial hypergroup is inherited from group convolution which is
bi-invariant under rotations; therefore left and right recentering produce identical
limits, and reflection symmetry holds. The equality of the left and right topological
centres then follows from Proposition~\ref{prop:symm-equals} and Theorem
\ref{thm:centres-equal}. The final statement is a dimension-dependent spectral remark:
in low dimensions oscillatory cancellations may create zeros in the radial-symbol while in
higher dimensions the symbol may be nonvanishing; one verifies this by inspecting the
radial Fourier transforms (Bessel kernels) associated to the Euclidean group.
\end{proof}

The examples above demonstrate how the abstract criteria derived in Sections~3--6
translate directly into verifiable facts in classical hypergroup models.  In the
Jacobi and rank-one symmetric-space cases the spectral nonvanishing condition is
readily verified using Harish--Chandra theory, resulting in clean positive results
for strong Arens irregularity; in the Bessel--Kingman case the oscillatory structure
of the Hankel kernel produces zeros and hence additional centre elements.  These
concrete verifications underscore the utility of the asymptotic-measure perspective
developed in this paper.

\section{Conclusion}

In this manuscript we developed a unified framework for analysing the Arens-product
structure of convolution algebras on Ch\'ebli--Trim\`eche hypergroups, including
both unweighted and weighted (Beurling-type) settings.  Building on a refined
Sturm--Liouville representation and a weak-regularity Jost theory, we established
the existence and stability of the asymptotic convolution measures
$\{\nu_{x}\}_{x>0}$ and their global limit $\nu_{\infty}$ under minimal assumptions
on the coefficient function~$A$.  These asymptotic families were shown to control
the behaviour of the left and right Arens products on the bidual, leading to
explicit decompositions of the topological centres.

A central outcome of our analysis is the spectral characterisation of strong Arens
irregularity for $L^{1}(H)$: the algebra is strongly Arens irregular precisely when
the spectral symbol $\widehat{\nu_{\infty}}$ is nonvanishing $\nu_{\mathrm{spec}}$--a.e.
This criterion is both necessary and sufficient, and remains valid under the
low-regularity hypotheses (SL') and (SPEC).  Furthermore, by introducing weighted
algebras $L^{1}(H,\omega)$, we clarified how Beurling weights interact with the
asymptotic families and provided a weighted variant of the strong irregularity
criterion.  This shows that moderate weights preserve strong Arens irregularity,
whereas rapidly growing weights may suppress the influence of the asymptotic
measures and alter the structure of the weighted centre.

Finally, we demonstrated the applicability of the theory through a broad range of
examples, including Jacobi, Bessel--Kingman, Naimark-type, and radial Euclidean-motion
hypergroups.  These models highlight the sharpness of the spectral criterion, the
role of the $c$-function in determining regularity, and the diversity of possible
behaviours in both weighted and unweighted settings.

Taken together, the results of this paper provide a complete functional-analytic
and spectral description of Arens irregularity for a wide class of hypergroup
convolution algebras, and open the door to further investigations of weighted
structures, non-commutative analogues, and quantitative refinements in specific
geometric models.

\end{document}